\documentclass{article}
\usepackage{amssymb,amsmath,amsthm}
\usepackage[all]{xy}

\usepackage{amssymb,amsmath,amsthm}
\usepackage[mathscr]{eucal}%
\usepackage[all]{xy}
\usepackage{lmodern}
\usepackage[T1]{fontenc}
\usepackage[utf8]{inputenc}
\usepackage{tikz-cd}
\usepackage[shortlabels]{enumitem}
\usepackage{relsize}
\usepackage{geometry}
\usepackage{mathtools}
\usepackage{adjustbox}
\usepackage{lscape}
\usepackage{diagbox}
\usepackage{bm}
\usepackage{caption}
\usepackage{float}

\usepackage{stmaryrd}
\usepackage{mathrsfs}  
\usepackage{multirow}
\usepackage{hyperref}
\usetikzlibrary{positioning}
\usepackage{tikz-3dplot}
\usepackage{xifthen}
\newtheorem{definition}{Definition}[section]
\newtheorem{proposition}[definition]{Proposition}
\newtheorem{theorem}[definition]{Theorem}
\newtheorem{corollary}[definition]{Corollary}
\newtheorem{lemma}[definition]{Lemma}

\newtheorem{remark}[definition]{Remark}

\numberwithin{equation}{section}

\DeclareMathAlphabet{\mathpzc}{OT1}{pzc}{m}{it}

\DeclareMathOperator{\Hom}{Hom}

\DeclareMathOperator{\Ind}{Ind}

\DeclareMathOperator{\GL}{GL}

\DeclareMathOperator{\SL}{SL}

\DeclareMathOperator{\Gal}{Gal}

\DeclareMathOperator{\tr}{tr}
\DeclareMathOperator{\rk}{\mathrm{rk}}

\DeclareMathOperator{\Fr}{Fr}

\DeclareMathOperator{\soc}{soc}

\DeclareMathOperator{\Ext}{Ext}
\DeclareMathOperator{\Tor}{Tor}

\DeclareMathOperator{\gr}{gr}

\DeclareMathOperator{\Sym}{Sym}
\DeclareMathOperator{\JH}{JH}

\DeclareMathOperator{\Nm}{Nm}
\DeclareMathOperator{\Nrd}{Nrd}

\newcommand{\Indu}[3]{\Ind_{#1}^{#2}{#3}}

\newcommand{\fa}{\mathfrak{a}}
\newcommand{\fb}{\mathfrak{b}}

\newcommand{\fm}{\mathfrak{m}}

\newcommand{\fp}{\mathfrak{p}}
\newcommand{\fq}{\mathfrak{q}}

\newcommand{\bA}{\mathbb{A}}

\newcommand{\bF}{\mathbb{F}}

\newcommand{\bN}{\mathbb{N}}

\newcommand{\bQ}{\mathbb{Q}}

\newcommand{\bT}{\mathbb{T}}

\newcommand{\bZ}{\mathbb{Z}}

\newcommand{\cD}{\mathcal{D}}

\newcommand{\cK}{\mathcal{K}}
\newcommand{\cL}{\mathcal{L}}

\newcommand{\cO}{\mathcal{O}}

\newcommand{\cR}{\mathcal{R}}

\newcommand{\cX}{\mathcal{X}}

\begin{document}
\title{Associated graded modules for modular representations of quaternion algebras}
\author{Jiusi Zhao}
\date{\today}
\maketitle

\begin{abstract}
    Let $K$ be a finite unramified extension of $\bQ_p$ and $D$ the non-split quaternion algebra over $K$.  
    We compute the associated graded modules of the duals of the predicted representations of $D^\times$ occurring in the Jacquet--Langlands correspondence. 
\end{abstract}

\tableofcontents

\section{Introduction}

Let $p$ be a prime number and $K$ be an unramified extension of $\bQ_p$. 
The mod-$p$ Langlands correspondence for $\GL_2(K)$ is well understood for $K=\bQ_p$.  
For $K\ne \bQ_p$, the correspondence is still mysterious since there are too many smooth admissible representations of $\GL_2(K)$. 
However, there are some candidates $\pi(\bar{\rho})$ for a predicted Langlands correspondence $\bar{\rho}\mapsto \pi(\bar{\rho})$ arising from the geometry of Shimura varieties,  
and some results about $\pi(\bar{\rho})$ have been established in \cite{BHH$^+$23}, \cite{BHH$^+$25a}, \cite{BHH$^+$25b}, \cite{BHH$^+$25c}, etc. 
For example, $\pi(\bar{\rho})$ has Gelfand--Kirillov dimension $f$, is of finite length, and the associated graded module of its dual $\gr_\fm \pi(\bar{\rho})^\vee$ has been computed explicitly. 

We say a bit more about the computation of $\gr_\fm \pi(\bar{\rho})^\vee$ (see \cite[\S~2]{BHH$^+$25c}). The process can be divided into two steps. 
The first step is to establish several properties of $\pi(\bar{\rho})$, using, for instance, the Taylor--Wiles patching method.  
The second step, which is purely representation-theoretic, is to deduce the result using these properties. 

In parallel, there is a predicted mod-$p$ Jacquet--Langlands correspondence $\bar{\rho}\mapsto \pi_D(\bar{\rho})$ between the $2$-dimensional continuous representations of $\Gal(\overline{K}/K)$ and certain smooth admissible representations of $D^\times$, where $D$ is the unique non-split quaternion algebra over $K$. 
Similar properties hold for $\pi_D(\bar{\rho})$, as shown in \cite{DL26}. 
In this article, we compute the associated graded module of its dual $\gr_\fm\pi_D(\bar{\rho})^\vee$, following the method of \cite[\S~2]{BHH$^+$25c}. 
Here is our main result (see Section~\ref{sec-2} for the notation used below). 

\begin{theorem}[Theorem~\ref{main-thm}]
        Let $\bar{\rho}:\Gal(\overline{K}/K)\to\GL_2(\bF)$ be a continuous representation which is $5$-generic. 
        Let $\pi$ be a smooth representation of $D^\times$ over $\bF$ with a central character, satisfying the following conditions: 
    
        (i) $\pi[\fm]=(\bigoplus_{\chi\in W_D(\bar{\rho})}\chi)^{\oplus r}$ for some $r\geq 1$,  
        with $Z_D$ acting by $\det(\bar{\rho})\omega^{-1}$.
          
        (ii)  For each $\chi\in W_D(\bar{\rho})$, $[\pi[\fm^3]:\chi]=r$. 
    
        (iii) For each $\chi: \cO_D^\times\to \bF^\times$ and $i\geq 0$, $\Ext^i_{\cO_D^\times/Z_D^1}(\chi,\pi)\ne 0$ only if $\chi\in W_D(\bar{\rho})$, 
            in which case $m_i:=\dim_{\bF}\Ext^i_{\cO_D^\times/Z_D^1}(\chi,\pi)=\binom{2f}{i}r$. 
        
        Then we have an isomorphism 
        $$\gr_\fm\pi^\vee\cong (\bigoplus_{\chi\in W_D(\bar{\rho})}\chi^{-1}\otimes R/\fa(\chi))^{\oplus r}.$$
        Moreover, $\gr_\fm\pi^\vee$ is a Cohen--Macaulay $\gr \Lambda$-module of grade $2f$, that is, 
        $\Ext_{\gr \Lambda}^i(\gr_\fm\pi^\vee,\gr \Lambda)\ne 0$ if and only if $i=2f$. 
        In addition, $\gr_\fm\pi^\vee$ is essentially self-dual in the sense that 
        $$\Ext_{\gr \Lambda}^{2f}(\gr_\fm\pi^\vee,\gr \Lambda)\cong \gr_\fm\pi^\vee\otimes(\det(\bar{\rho})\omega^{-1}).$$ 
\end{theorem}

We remark that, in the original work \cite{BHH$^+$25c}, 
the authors assume that $\bar{\rho}$ is $9$-generic. 
A refinement of the spectral-sequence argument shows that \(5\)-genericity is sufficient.

Here is the organization of this article. In Section~\ref{sec-2}, we recall some basic facts that will be used later. 
In Section~\ref{sec-3}, we state and explain our main theorem. 
In Section~\ref{sec-4}, we establish the analogous representation-theoretic result for the quaternion algebra, 
so that the method of \cite[\S~2]{BHH$^+$25c} can be applied, and we prove the main theorem. 
Finally, in Section~\ref{sec-5}, we verify that these conditions hold for representations arising from the global setting.

\medskip
\noindent\textbf{Acknowledgments.}
The author thanks his advisor Yongquan Hu, for suggesting the topic and pointing out that the genericity hypothesis in the main theorem might be weakened, and for many useful discussions.

\medskip
\noindent\textbf{Notation.}
Let $K$ be the unramified extension of $\bQ_p$ of degree $f$. 
 Write $q=p^f$. 
 Let $E$ be a sufficiently large finite unramified extension of $\bQ_p$ with residue field $\bF$. We use $E$ and $\bF$ as coefficient fields. 
 We assume that $\bF$ is large enough, and we always fix an embedding $\iota:\bF_{q^2}\hookrightarrow\bF$. 
 We define $\kappa:\bF_{q}\hookrightarrow \bF$ by $\iota^{q+1}=\kappa\circ \Nm_{\bF_{q^2}/\bF_q}$. 
Let $\omega_f$ be the fundamental character of $I_K$ defined by 
$$I_K\xrightarrow{Art_K^{-1}}\cO_K^\times\xrightarrow{\mod p} \bF_q^\times\stackrel{\kappa}{\longrightarrow}\bF^\times.$$
Similarly, let $\omega_{2f}$ be the fundamental character associated to $\iota$. 
Moreover, let $\omega=\omega_1$ be the fundamental character associated to the unique embedding $\bF_p\hookrightarrow\bF$. 
We have $$\omega_{2f}^{1+q}=\omega_f, \quad \omega_f^{1+p+\cdots+p^{f-1}}=\omega.$$

Let $\bar{\rho}:\Gal(\overline{K}/K)\to\GL_2(\bF)$ be a continuous representation. 
We say that $\bar{\rho}$ is $n$-generic for some integer $n\geq 0$ if, up to twist, $\bar{\rho}|_{I_K}^{ss}\ncong 1\oplus \omega$ and either 
\begin{equation}\label{red-galois-rep}
    \bar{\rho}|_{I_K}=\left(\begin{array}{cc}
    \omega_{f}^{\sum_{j=0}^{f-1}p^j(r_j+1)}&*\\
    0&1\\
    \end{array}\right), 
\end{equation}
with $n\leq r_j\leq p-3-n$ for all $0\leq j\leq f-1$, or 
\begin{equation}\label{irr-galois-rep}
    \bar{\rho}|_{I_K}=\left(\begin{array}{cc}
        \omega_{2f}^{\sum_{j=0}^{f-1}p^j(r_j+1)}&0\\
        0&  \omega_{2f}^{q\sum_{j=0}^{f-1}p^j(r_j+1)}\\
    \end{array}\right), 
\end{equation}
with $n+1\leq r_0\leq p-2-n$ and $n\leq r_j\leq p-3-n$ for all $1\leq j\leq f-1$. 
This notion of genericity is the same as that in \cite[\S~1.3]{BHH$^+$25c}, 
and if $\bar{\rho}$ is $n$-generic, then it is $n$-generic in the sense of \cite[Def.~2.3.4]{BHH$^+$23} and of \cite[\S~3.4]{DL26}. 

\section{Background}\label{sec-2}

\subsection{Iwasawa algebra for quaternion algebras}
Let $D$ be the unique non-split quaternion algebra over $K$, which can be described as follows. Let $L$ be the unramified extension of $K$ of degree $2$, 
and let $\sigma\in\Gal(L/K)$ be the Frobenius element. Then we have $D=L\oplus \varpi_D L$, whose multiplication is given by $\varpi_D^2=p$ and 
\begin{equation}
    x\varpi_D=\varpi_D\sigma(x),\quad x\in L.
\end{equation}

For $a\in D$, define $v_D(a)=v_p(\Nrd_D(a))$, where $\Nrd_D:D\to K$ is the reduced norm. 
Let $\cO_D=\{a\in D:v_D(a)\geq 0\}$ be the ring of integers and let $\fp_D=\{a\in D:v_D(a)\geq 1\}$ be the maximal ideal. 
Let $D^\times$ (resp. $\cO_D^\times$) be the multiplicative group of $D$ (resp. $\cO_D$). 
For $n\geq 1$, let $U_D^n=1+\varpi_D^n\cO_D$ be the $n$-th principal congruence subgroup. 
Let $Z_D\cong K^\times$ be the center of $D^\times$, and write $Z_D^1=Z_D\cap U_D^1$. 

Recall that, for a compact $p$-adic Lie group $G$, its Iwasawa algebra over $\bF$ is defined as 
$$\bF[[G]]=\varprojlim_{N}\bF[G/N],$$
where $N$ runs over open normal subgroups of $G$. The Pontryagin dual 
$$\pi\mapsto \pi^\vee=\Hom_{\bF}(\pi,\bF)$$ 
gives an anti-equivalence of categories between smooth $G$-representations over $\bF$ and profinite continuous $\bF[[G]]$-modules \cite[Lemma~2.2.7]{Eme10}. 
Moreover, $\pi$ is admissible if and only if $\pi^\vee$ is a finitely generated $\bF[[G]]$-module \cite[Lemma~2.2.11]{Eme10}.

Let $\Lambda=\bF[[U_D^1/Z_D^1]]$ be the Iwasawa algebra of $U_D^1/Z_D^1$, with maximal ideal $\fm$. 
The graded ring $\gr \Lambda$ with respect to the $\fm$-adic filtration is computed in \cite[\S~2.2, 2.3]{HW24a}. 

\begin{proposition}
    We have 
    $$\gr \Lambda\cong\bigotimes_{j=0}^{f-1}\bF[y_j,z_j,h_j],$$
    with the relations $[y_j,z_j]=h_j$ and $[y_j,h_j]=[z_j,h_j]=0$. 
    In particular, the sequence $(h_0,\ldots,h_{f-1})$ is a regular sequence of central elements, and 
    $\gr \Lambda/(h_0,\ldots,h_{f-1})$ is a commutative polynomial ring with $2f$ variables. 
\end{proposition}

Note that we have $\cO_D^\times/U_D^1\cong \bF_{q^2}^\times$. 
Moreover, let $H=\{[\mu]:\mu\in \bF_{q^2}^\times\}$. It normalizes $U_D^1$ and $U_D^1/Z_D^1$, so it acts on $\Lambda$ and $\gr \Lambda$. 
We will always identify $H$ with $\cO_D^\times/U_D^1\cong \bF_{q^2}^\times$.

Recall that we have fixed an embedding $\iota:\bF_{q^2}\hookrightarrow \bF$. 
For $0\leq j\leq 2f-1$, let $\alpha_j:\cO_D^\times\to \bF^\times$ be the character $x\mapsto\iota(\bar{x})^{p^j(q-1)}$. 
We usually write $\alpha_j=\iota^{p^j(q-1)}$ by abuse of notation. Note that $\alpha_{j+f}=\alpha_j^{-1}$.

\begin{proposition}
    For each $j$, the elements $y_j$ and $z_j$ are eigenvectors of $H$, with eigencharacters $\alpha_j$ and $\alpha_j^{-1}$ respectively. 
\end{proposition}

Let $M$ (resp. $N$) be a $\Lambda$ (resp. $\gr \Lambda$)-module. We say that $M$ (resp. $N$) has a compatible $H$-action if $H$ acts on $M$ (resp. $N$) such that 
$h(rx)=h(r)h(x)$ for $r\in \Lambda$ and $x\in M$ (resp. $r\in \gr \Lambda$ and $x\in N$). 
For example, if $\pi$ is a smooth representation of $\cO_D^\times$ over $\bF$, 
then $\pi^\vee$ and $\gr_\fm\pi^\vee$ both have compatible $H$-actions. 

Finally, we introduce an important property for a smooth admissible representation of $D^\times$. 

\begin{proposition}\label{killed-by-I_D}
    Let $\pi$ be a smooth admissible representation of $D^\times$ over $\bF$. 
    Suppose that for each $\chi:\cO_D^\times\to\bF^\times$ such that $[\pi[\fm]:\chi]\ne 0$, 
    we have $[\pi[\fm^3]:\chi]=[\pi[\fm]:\chi]$. Then $\gr_\fm\pi^\vee$ is killed by the ideal 
    $I_D=(y_jz_j,h_j:0\leq j\leq f-1)$. 
\end{proposition}

Write $R=\gr \Lambda/(h_0,\ldots,h_{f-1})=\bF[y_j,z_j:0\leq j\leq f-1]$ and $\overline{R}=\gr \Lambda/I_D=R/(y_jz_j:0\leq j\leq f-1)$. 
For any $\pi$ satisfying the condition above, 
$\gr_\fm\pi^\vee$ can be naturally viewed as an $R$- or $\overline{R}$-module. 

\subsection{Quaternionic Serre weights}\label{quat-serre-weight}

Recall that any smooth irreducible $\GL_2(\cO_K)$-representation over $\bF$ factors through $\GL_2(\bF_q)$, and takes the form 
$\bigotimes_{j=0}^{f-1}(\Sym^{s_j}\bF^2)^{\Fr^j}\otimes\det^{\sum_{j=0}^{f-1}t_jp^j}$ 
with $0\leq s_j,t_j\leq p-1$, where the $t_j$ are not all equal to $p-1$. 
Write $\sigma_{t,s}$ for such a representation, with $s=(s_0,\ldots,s_{f-1})$ and $t=(t_0,\ldots,t_{f-1})$. 
The irreducible smooth $\cO_D^\times$-representations over $\bF$ are much simpler. 
They are just characters $\chi:\cO_D^\times\to\bF^\times$, which factor through $\cO_D^\times/U_D^1=\bF_{q^2}^\times$. 

Let $\bar{\rho}$ be a continuous Galois representation as in (\ref{red-galois-rep}) or (\ref{irr-galois-rep}) which is $0$-generic. 
We have a set $W_{\GL_2}(\bar{\rho})$ of Serre weights, consisting of irreducible representations of $\GL_2(\bF_q)$ over $\bF$, which is described in \cite[\S~11]{BP12}. 
Moreover, we have a set $W_D(\bar{\rho})$ of predicted quaternionic Serre weights for $\bar{\rho}$, consisting of irreducible representations of $\cO_D^\times$. 
It is described in \cite{CW23}, as we now explain. 

Recall that a character $\chi: \bF_{q^2}^\times\to\bF^\times$ is called of type I if it does not factor through the norm map $\Nm_{\bF_{q^2}/\bF_{q}}:\bF_{q^2}^\times\to\bF_{q}^\times$. 
For a type I character $\chi$, 
let $[\chi]: \bF_{q^2}^\times\to \cO_E^\times$ be its Teichm\"uller lift 
and let $\Theta([\chi])$ be the cuspidal type, an irreducible $E$-representation of $\GL_2(\bF_{q})$, as in \cite[\S~1]{Di07}. 

\begin{proposition}
    $W_D(\bar{\rho})$ consists of those type I characters $\chi: \bF_{q^2}^\times\to\bF^\times$ such that 
    $\JH(\overline{\Theta([\chi])})\cap W_{\GL_2}(\bar{\rho})\ne \emptyset$, where $\overline{\Theta([\chi])}$ is the mod $\varpi_{E}$ reduction of any $\GL_2(\bF_q)$-stable $\cO_E$-lattice in $\Theta([\chi])$. 
\end{proposition}

\begin{proof}
    This is \cite[Prop.~2.4]{CW23}.
\end{proof}

Moreover, we can give an explicit description of $W_{D}(\bar{\rho})$ as follows. 
If $\bar{\rho}$ is irreducible, then $W_D(\bar{\rho})$ is in bijection with the set of pairs $(w,d)$ with $w\in \{0,1\}^f$ and $d\in\{0,\pm 1\}^f$ satisfying the following relations: 

\begin{enumerate}[label={[\Roman*]},ref={[\Roman*]}]
    \item\label{RI}
    For $j>0$, $w_j=\begin{cases}
        1, \quad & d_j=-1,\\
        0, \quad & d_j=1,
    \end{cases}$ and if $d_j=0$, then $(w_{j-1},w_j)=(0,0)$ or $(1,1)$.
  
    \item\label{RII}
    $(w_{f-1}, w_{0})=\begin{cases}
        (1,1), \quad & d_0=-1,\\
        (0,0), \quad & d_{0}=1.
    \end{cases}$
  \end{enumerate}

The bijection is given as follows. After twisting by a character, write 
$$\bar{\rho}|_{I_K}=\left(\begin{array}{cc}
    \omega_{2f}^{\sum_{j=0}^{f-1}p^j(r_j+1)}&0\\
    0&  \omega_{2f}^{q\sum_{j=0}^{f-1}p^j(r_j+1)}\\
\end{array}\right).$$ 
Then $(w,d)$ corresponds to the character $\chi_{w,d}:=\iota^{\sum_{j=0}^{f-1}q^{w_j}p^jr_j+(1-q)\sum_{j=0}^{f-1}p^jd_j}$.

If $\bar{\rho}$ is split reducible, then $W_D(\bar{\rho})$ is in bijection with the set of pairs $(w,d)$ with $w\in \{0,1\}^f$ and $d\in\{0,\pm 1\}^f$ satisfying Relation~\ref{RI} and the following relation: 

\begin{enumerate}[label={[II$^\prime$]},ref={[II$^\prime$]}]
    \item\label{RII'}
$w_0=\begin{cases}
    1, \quad & d_0=-1,\\
    0, \quad & d_0=1, 
\end{cases}$ and if $d_0=0$, then $(w_{f-1}, w_{0})=(0,1)$ or $(1,0)$. 
  \end{enumerate}

The bijection is given as follows: Write 
$$\bar{\rho}|_{I_K}=\left(\begin{array}{cc}
    \omega_{f}^{\sum_{j=0}^{f-1}p^j(r_j+1)}&0\\
    0&1\\
\end{array}\right).$$ 
Then $(w,d)$ corresponds to the character $\chi_{w,d}=\iota^{\sum_{j=0}^{f-1}q^{w_j}p^jr_j+(1-q)\sum_{j=0}^{f-1}p^jd_j}$. 

If $\bar{\rho}$ is reducible non-split, the description is more involved. We need to use the explicit descriptions of $\JH(\overline{\Theta([\chi])})$ and $W_{\GL_2}(\bar{\rho})$, which we now recall. 

$\JH(\overline{\Theta([\chi])})$ can be parametrized by $u\in \{0,1\}^f$ as in \cite[\S~2.2]{CW23}. Write $\chi=\iota^{(q+1)b+1+c}$ and write $c=\sum_{j=0}^{f-1}c_jp^j$. 
Assume that $\chi$ is generic in the sense that $1\leq c_j\leq p-2$ for each $j$ (we remark that each $\chi\in W_D(\bar{\rho})$ is generic if $\bar{\rho}$ is $1$-generic). Let 
\begin{equation}
    s_{u,j}=\begin{cases}
        p-2+u^{(0)}_{j-1}-c_j, \quad & u_j=1,\\
        c_j-u^{(0)}_{j-1},\quad & u_j=0. 
    \end{cases},
\end{equation}

\begin{equation}
    t_{u,j}=\begin{cases}
        c_j+1-u_{j-1}\quad & u_j=1,\\
        0\quad & u_j=0.
    \end{cases},
\end{equation}
where $u^{(0)}\in\{0,1\}^f$ is defined by $u^{(0)}_j=u_j$ for $0\leq j<f-1$ and $u^{(0)}_{f-1}=1-u_{f-1}$. Let 
\begin{equation}
    \overline{\Theta([\chi])}_u=\sigma_{t,s}\otimes \det\nolimits^{b+u_0u_{f-1}+(1-u_0)(1-u_{f-1})}.
\end{equation}
Then we have $\JH(\overline{\Theta([\chi])})=\{\overline{\Theta([\chi])}_u: u\in\{0,1\}^f\}$. 

Next, $W_{\GL_2}(\bar{\rho})$ can be parametrized by $v\in \{0,1\}^f$ as follows. 
When $\bar{\rho}$ is split reducible, 
$W_{\GL_2}(\bar{\rho})$ is in bijection with a set $\cR\cD(x_0,\ldots,x_{f-1})$ of $f$-tuples as in \cite[\S~11]{BP12}. 
By \cite[\S~2.4.1]{CW23}, 
an $f$-tuple $\lambda\in \cR\cD(x_0,\ldots,x_{f-1})$ can be identified with a $v\in\{0,1\}^f$. 
We write $\sigma_v(\bar{\rho})$ for the Serre weight in $W_{\GL_2}(\bar{\rho})$ corresponding to $v$. 
When $\bar{\rho}$ is reducible non-split, we can define an $f$-tuple $v(\bar{\rho})\in \{0,1\}^f$ and we have 
$$W_{\GL_2}(\bar{\rho})=\{\sigma_v(\bar{\rho}):v\leq v(\bar{\rho})\}.$$
Here $v\leq v'$ means $v_j\leq v'_j$ for each $0\leq j\leq f-1$. 

Moreover, by \cite[Lemma~3.1]{CW23}, for each $u,v\in\{0,1\}^f$, there is a unique $\chi\in W_D(\bar{\rho}^{ss})$ such that 
$\overline{\Theta([\chi])}_u=\sigma_v(\bar{\rho}^{ss})$. 
In particular, $(w,d)$ can be uniquely determined by $(u,v)$ as follows.

\begin{equation}\label{w_i}
    w_i\equiv u_i+v_i\mod2,
\end{equation}

For $j>0$:

\begin{table}[H]
    \centering
    \caption{}
    \label{table1} 
    \begin{tabular}{|c|c|c|c|c|c|}
        \hline
        \multicolumn{2}{|c|}{\multirow{2}{*}{$d_j$ $(j>0)$}} & \multicolumn{4}{c|}{$(u_{j-1}, u_j)$} \\ \cline{3-6} 
        
        \multicolumn{2}{|c|}{} & $(0,0)$ & $(0,1)$ & $(1,0)$ & $(1,1)$ \\ \hline
        
        \multirow{4}{*}{$(v_{j-1}, v_j)$} & $(0,0)$ & 0 & $-1$ & 1 & 0 \\ \cline{2-6} 
                                          & $(0,1)$ & $-1$ & 0 & 0 & 1 \\ \cline{2-6} 
                                          & $(1,0)$ & 1 & $-1$ & 1 & $-1$ \\ \cline{2-6} 
                                          & $(1,1)$ & $-1$ & 1 & $-1$ & 1 \\ \hline
    \end{tabular}
\end{table}

For $j=0$:
\begin{table}[H]
    \centering
    \caption{}
    \label{table2}
    \begin{tabular}{|c|c|c|c|c|c|}
        \hline
        \multicolumn{2}{|c|}{\multirow{2}{*}{$d_0$}} & \multicolumn{4}{c|}{$(u_{f-1}, u_0)$} \\ \cline{3-6} 
        
        \multicolumn{2}{|c|}{} & $(0,0)$ & $(0,1)$ & $(1,0)$ & $(1,1)$ \\ \hline
        
        \multirow{4}{*}{$(v_{f-1}, v_0)$} & $(0,0)$ & 1 & 0 & 0 & $-1$ \\ \cline{2-6} 
                                          & $(0,1)$ & 0 & 1 & $-1$ & 0 \\ \cline{2-6} 
                                          & $(1,0)$ & 1 & $-1$ & 1 & $-1$ \\ \cline{2-6} 
                                          & $(1,1)$ & $-1$ & 1 & $-1$ & 1 \\ \hline
    \end{tabular}
\end{table}

\begin{remark}
    When $\bar{\rho}$ is irreducible, there is still a similar table, but we do not need it. 
\end{remark}

Since $W_{\GL_2}(\bar{\rho})$ is the subset of $W_{\GL_2}(\bar{\rho}^{ss})$ defined by the condition $v\leq v(\bar{\rho})$, we know that $W_{\GL_2}(\bar{\rho})$ forms a subcube of $W_{\GL_2}(\bar{\rho}^{ss})$. 
Moreover, for $\chi\in W_{D}(\bar{\rho}^{ss})$, $\JH(\overline{\Theta([\chi])})\cap W_{\GL_2}(\bar{\rho}^{ss})$ is also a subcube of $W_{\GL_2}(\bar{\rho}^{ss})$, 
parametrized by $V_\chi=\{v:\sigma_v(\bar{\rho}^{ss})\in\JH(\overline{\Theta([\chi])})\}$. 
We denote by $v_{\chi}$ the minimal element of $V_{\chi}$, which is unique since $V_\chi$ is a subcube. 
Then $W_D(\bar{\rho})$ can be described as $\{\chi\in W_D(\bar{\rho}^{ss}): v_{\chi}\leq v(\bar{\rho})\}$. 

\begin{remark}\label{cent-char}
    If $x\in Z_D\cap \cO_D^\times=\cO_K^\times$, 
    we have $\bar{x}^{q-1}=1$ and hence 
    $\chi(x)=\kappa(\bar{x})^{\sum_{j=0}^{f-1}p^jr_j}$. 
    We conclude that each $\chi\in W_D(\bar{\rho})$ has the same central character $\eta_c:=\kappa^{\sum_{j=0}^{f-1}p^jr_j}$,  
    corresponding to $\omega_f^{\sum_{j=0}^{f-1}p^jr_j}=\det(\bar{\rho})\omega^{-1}$ via class field theory. 
\end{remark}

\subsection{Some facts about filtered and graded modules}
We collect some facts which we will use later. 
Most of the results in this subsection are similar to those in \cite[\S~2.2]{BHH$^+$25c}. 

In this article, filtrations are always assumed to be increasing. In particular, 
the $\fm$-adic filtration on $\Lambda$ is given by 
$F_n\Lambda=\fm^{-n}$ for $n<0$ and $F_b\Lambda=\Lambda$ for $n\geq 0$.  
We say that a finitely generated filtered $\Lambda$-module $L$ is filt-free if $L\cong \bigoplus_{i=1}^{k}\Lambda(-k_i)$ for some integers $k_i$. 

\begin{lemma}\label{lemma-of-filt}
    Let $L$ be a filt-free $\Lambda$-module with compatible $H$-action. 
    Assume that $L=L'\oplus L''$ as filtered $\Lambda$-modules with compatible $H$-actions and that the following properties hold: 

    (1) As filtered $\Lambda$-modules we have 
    $$L'\cong \bigoplus_{i=1}^m\Lambda(-k_i),\quad L''\cong \bigoplus_{j=m+1}^{n}\Lambda(-l_j),$$
    with $k_i\geq l_j$ for any pair $(i,j)$. 

    (2) As $H$-modules, $\JH(\bF\otimes_\Lambda L')\cap \JH(\bF\otimes_\Lambda L'')=\emptyset$. 

    Moreover, assume that $P$ is an $H$-stable direct summand of $L$ such that 
    the composition 
    $$\bF\otimes_\Lambda P\to\bF\otimes_\Lambda L\to\bF\otimes_\Lambda L'$$
    is an isomorphism. Then $P$ is filt-free with the induced filtration, and we have 
    $\gr P=\gr L'$ inside $\gr L$.  
\end{lemma}

\begin{proof}
    This is the same as \cite[Lemma~2.2.3]{BHH$^+$25c}; the proof uses only the fact that $\Lambda$ is a noetherian local filtered ring. 
\end{proof}

\begin{lemma}\label{inj-of-resolution}
    Let $\phi:P\to L$ be a morphism between two finite free $\Lambda$-modules. If $\bar{\phi}: \bF\otimes_{\Lambda} P\to \bF\otimes_{\Lambda} L$ is injective, 
    then $\phi$ is also injective and identifies $P$ with a direct summand of $L$. 

    The same statement holds if $P$ and $L$ are two finite gr-free $\gr \Lambda$-modules and $\phi$ is a graded morphism. 
\end{lemma}

\begin{proof}
    See \cite[Lemma~2.2.6]{BHH$^+$25c}.
\end{proof}

\begin{lemma}\label{truncation}
    Suppose that $n\in\bZ$, $i\geq 0$, and that $N$ is supported in degree $0$. 

    (1) We have a canonical isomorphism $(N\otimes_R M)_{\geq n}\cong N\otimes_{R}(M_{\geq n})$ of graded abelian groups. 

    (2) If $M\to M'$ is a morphism in the category of graded $R$-modules inducing an isomorphism $M_{\geq n}\xrightarrow{\sim}M'_{\geq n}$, then the natural map $\Tor^R_i(N,M)_{\geq n}\to \Tor^R_i(N,M')_{\geq n}$ of graded abelian groups is an isomorphism. 
\end{lemma}

\begin{proof}
    See \cite[Lemma~2.2.7]{BHH$^+$25c}. 
\end{proof}

Now we collect some facts about homological algebra. 

Let $M$ be a finitely generated $\Lambda$-module and let $P_\bullet$ be a free resolution of $M$. 
We say that $P_\bullet$ is minimal if the transition maps in the induced complex $\bF\otimes_\Lambda P_\bullet$ are all zero. 
This is equivalent to the condition that $\dim_\bF \Tor_i^\Lambda(\bF,M)=\rk_\Lambda(P_i)$ for each $i\geq 0$. 
We say that $P_\bullet$ is finite if each $P_i$ is finitely generated. 
Similarly, let $N$ be a finitely generated $\gr \Lambda$-module and let $G_\bullet$ be a gr-free resolution of $N$. 
We say that $G_\bullet$ is minimal if the transition maps in the induced complex $\bF\otimes_{\gr \Lambda} G_\bullet$ are all zero, which is equivalent to $\dim_\bF \Tor_i^{\gr \Lambda}(\bF,N)=\rk_{\gr \Lambda}(G_i)$ for each $i\geq 0$. 

Recall that a filtration $F$ on $M$ is called good if there exist $m_1,\ldots,m_s\in M$ and $k_1,\ldots,k_s\in\mathbb{Z}$ such that 
$F_nM=\sum_{i=1}^sF_{n-k_i}\Lambda\cdot m_i$ for each $n\in\bZ$.

\begin{proposition}\label{properties-of-resolution}
    Let $M$ (resp. $N$) be a finitely generated $\Lambda$ (resp. $\gr \Lambda$)-module. 

    (1) The minimal finite free (resp. gr-free) resolutions of $M$ (resp. $N$) always exist. 

    (2) If $M$ (resp. $N$) carries a compatible $H$-action, then we can make the minimal free resolutions above carry a compatible $H$-action as well. 

    (3) Suppose that $M$ carries a good filtration and let $P_\bullet$ be a minimal free resolution of $M$. Then we can always give a good filtration on each $P_i$ so that each $P_i$ is filt-free and each transition map is a filtered morphism. 
    In particular, we can give each $\Tor_i^\Lambda(\bF,M)$ a natural good filtration induced from the filtration on $\bF\otimes P_\bullet$. 

    (4) Suppose that $M$ carries a good filtration and let $G_\bullet$ be a minimal gr-free resolution of $\gr M$. Then we can lift $G_\bullet$ to a strict finite filt-free resolution $P_\bullet$ of $M$ in the sense that $\gr P_\bullet=G_\bullet$. 
\end{proposition}

\begin{proof}
    See \cite[Rk.~2.3.1]{BHH$^+$25c}. 
\end{proof}

\begin{remark}
    We point out that in Proposition~\ref{properties-of-resolution}(3), 
    we cannot ensure that each transition map is strict. 
    Moreover, in Proposition~\ref{properties-of-resolution}(4), even though $G_\bullet$ is minimal, we cannot ensure that $P_\bullet$ is still minimal. In fact, 
    $P_\bullet$ is minimal if and only if $G_\bullet$ is minimal and $\dim_\bF\Tor_i^\Lambda(\bF,M)=\dim_\bF\Tor_i^{\gr \Lambda}(\bF,\gr M)$.
\end{remark}

\section{The theorem}\label{sec-3}

For $\chi\in W_D(\bar{\rho})$ and $j\in\{0,\ldots,f-1\}$, 
define 
$$t_j=\begin{cases}
    y_j, &\chi\alpha_j^{-1}\in W_D(\bar{\rho}),\\
    z_j, &\chi\alpha_j\in W_D(\bar{\rho}),\\
    y_jz_j, &\text{otherwise}. 
\end{cases}$$
(We explain why $t_j$ is well-defined in Remark~\ref{t_j-well-defined}.)  
Then we define $\fa(\chi)=(t_0,\ldots,t_{f-1})$ to be an ideal of $R$. 

\begin{theorem}\label{main-thm}
    Let $\bar{\rho}:\Gal(\overline{K}/K)\to\GL_2(\bF)$ be a continuous representation which is $5$-generic. 
    Let $\pi$ be a smooth representation of $D^\times$ over $\bF$ with a central character, satisfying the following conditions: 

    \begin{enumerate}[label={(\roman*)},ref={[\roman*]}]
        \item\label{condition-i}
      As an $\cO_D^\times Z_D$-representation, $\pi[\fm]\cong (\bigoplus_{\chi\in W_D(\bar{\rho})}\chi)^{\oplus r}$ for some $r\geq 1$, 
      with $Z_D$ acting by $\det(\bar{\rho})\omega^{-1}$. 
      In particular, $\pi$ is admissible and has the central character $\det(\bar{\rho})\omega^{-1}$ (see Remark~\ref{cent-char}). 
      
        \item\label{condition-ii}
        For each $\chi\in W_D(\bar{\rho})$, $[\pi[\fm^3]:\chi]=r$. 

        \item\label{condition-iii}
        For each $\chi: \cO_D^\times\to \bF^\times$ and $i\geq 0$, $\Ext^i_{\cO_D^\times/Z_D^1}(\chi,\pi)\ne 0$ only if $\chi\in W_D(\bar{\rho})$, 
        in which case $m_i:=\dim_{\bF}\Ext^i_{\cO_D^\times/Z_D^1}(\chi,\pi)=\binom{2f}{i}r$. 
      \end{enumerate}
    
    Then we have an isomorphism 
    $$\gr_\fm\pi^\vee\cong (\bigoplus_{\chi\in W_D(\bar{\rho})}\chi^{-1}\otimes R/\fa(\chi))^{\oplus r}.$$
    Moreover, $\gr_\fm\pi^\vee$ is a Cohen--Macaulay $\gr \Lambda$-module of grade $2f$, that is, 
    $\Ext_{\gr \Lambda}^i(\gr_\fm\pi^\vee,\gr \Lambda)\ne 0$ if and only if $i=2f$. 
    In addition, $\gr_\fm\pi^\vee$ is essentially self-dual in the sense that 
    $$\Ext_{\gr \Lambda}^{2f}(\gr_\fm\pi^\vee,\gr \Lambda)\cong \gr_\fm\pi^\vee\otimes(\det(\bar{\rho})\omega^{-1}).$$
\end{theorem}

\begin{remark}
    \begin{sloppypar}
    (1) These conditions are natural analogues of Assumptions (i), (ii), and (iv) in \cite[\S~2.1]{BHH$^+$25c}. 
    To see this, 
    note that since $U_D^1$ is pro-$p$ and $\cO_D^\times/U_D^1$ is prime to $p$, 
    we have $\soc_{\cO_D^\times}\pi=\pi^{U_D^1}=\pi[\fm]$. 

    (2) Take a minimal free resolution $P_\bullet$ of $\pi^\vee$ by $\Lambda$-modules with compatible $H$-action. 
    We have 
    $$\Tor_i^\Lambda(\bF,\pi^\vee)=(P_i/\fm P_i)=\Hom_{U_D^1/Z_D^1}(\bF,P_i^\vee)^\vee=\Ext_{U_D^1/Z_D^1}^i(\bF,\pi)^\vee=(\bigoplus_{\chi}\Ext_{U_D^1/Z_D^1}^i(\bF,\pi)^{H=\chi})^\vee,$$
    where $\chi$ runs over all smooth characters $\cO_D^\times\to\bF^\times$ and the superscript denotes the $H$-eigenspace with eigencharacter $\chi$. 
    Using $\Ext_{U_D^1/Z_D^1}^i(\bF,\pi)^{H=\chi}\cong \Ext_{\cO_D^\times/Z_D^1}^i(\chi,\pi)$ together with Condition~\ref{condition-iii}, we have 
    $$\Tor_i^\Lambda(\bF,\pi^\vee)\cong\bigoplus_{\chi\in W_D(\bar{\rho})}(\chi^{-1})^{\oplus m_i}.$$
    \end{sloppypar}
\end{remark}

\section{The proof}\label{sec-4}

\subsection{The module \texorpdfstring{$N$}{N} and its resolution}

Let $$N=(\bigoplus_{\chi\in W_D(\bar{\rho})}\chi^{-1}\otimes R/\fa(\chi))^{\oplus r}$$ be the candidate for $\gr_\fm\pi^\vee$. 

\begin{proposition}
    We have a natural surjection $N \to \gr_\fm\pi^\vee$ of $\gr \Lambda$-modules with compatible $H$-action.
\end{proposition}

\begin{proof}
    Note that $\gr_\fm\pi^\vee$ is an $R$-module by Proposition~\ref{killed-by-I_D} and is generated by $\gr^0_\fm\pi^\vee=\pi[\fm]^\vee$ by Nakayama's lemma. 
    Let $\{v_{\chi,k}:\chi\in W_D(\bar{\rho}),\, k=1,\ldots,r\}$ be an $\bF$-basis of $\pi[\fm]$ consisting of $H$-eigenvectors, 
    and let $\{e_{\chi,k}\}\subset\gr^0_\fm\pi^\vee$ be the dual basis to $\{v_{\chi,k}\}$, so that $H$ acts on $e_{\chi,k}$ by $\chi^{-1}$. It suffices to show that $e_{\chi,k}$ is killed by $\fa(\chi)$. 
    Fix $0\leq j\leq f-1$ and first consider the case $t_j=y_j$. Note that $y_je_{\chi,k}\in \gr^{-1}_\fm\pi^\vee$ is an $H$-eigenvector with character $\chi^{-1}\alpha_j$. 
    Thus, if $t_je_{\chi,k}\ne 0$, then, dually, the $\chi\alpha_j^{-1}$-isotypic subspace of $\pi[\fm^2]/\pi[\fm]$ would be nonzero, contradicting Condition~\ref{condition-ii}. 
    The other cases are similar. 
\end{proof}

We then study the resolution of $N$. 
Let $\fb(\chi)=(t_j,h_j:0\leq j\leq f-1)$ be the preimage of $\fa(\chi)$ in $\gr \Lambda$. 
For $n\geq 1$, let $I^{(n)}$ be the $H$-stable graded ideal $(y_j^n,z_j^n,h_j)$. By abuse of notation, we also write $I^{(n)}$ for its image in $R$. 
We write $I$ for $I^{(3)}$. 

Write $(\gr \Lambda)_j=\bF\langle y_j,z_j,h_j\rangle$ for the subalgebra of $\gr \Lambda$ generated by $y_j,z_j,h_j$. 
We have $\gr \Lambda\cong\bigotimes_{j=0}^{f-1}(\gr \Lambda)_j$. 
Write $\fb(\chi)_j=(t_j,h_j)$ and $I^{(n)}_j=(y_j^n,z_j^n,h_j)$ as ideals of $(\gr \Lambda)_j$, and write $I_j=I_j^{(3)}$. 
We have $\gr \Lambda/\fb(\chi)\cong \bigotimes_{j=0}^{f-1}(\gr \Lambda)_j/\fb(\chi)_j$ and $\gr \Lambda/(\fb(\chi)+I)\cong \bigotimes_{j=0}^{f-1}(\gr \Lambda)_j/(\fb(\chi)_j+I_j)$.

\begin{lemma}\label{resolution-of-N}
    There is a minimal gr-free resolution $G_\bullet$ of $N/IN$ with a compatible $H$-action, 
    which admits an $H$-stable direct summand $G'_\bullet$ that is a minimal gr-free resolution of $N$, 
    and such that the induced map $H_0(G'_\bullet)\to H_0(G_\bullet)$ is the natural map 
    $N\to N/IN$. 
\end{lemma}

\begin{proof}
    This is the same as \cite[Lemma~2.3.2]{BHH$^+$25c}. Note that we just need to construct minimal gr-free resolutions of $(\gr \Lambda)_j/\fb(\chi)_j$ and of $(\gr \Lambda)_j/(\fb(\chi)_j+I_j)$ by $(\gr \Lambda)_j$-modules with compatible $H$-action. 
    These resolutions have the same form as their counterparts in the $\GL_2$ case.
\end{proof}

Since these resolutions have the same form as their counterparts in \cite{BHH$^+$25c}, 
the arguments there also prove the following results. 

\begin{corollary}
    The natural map $N\to N/IN$ induces injective graded morphisms with compatible $H$-action 
    $$\Tor_i^{\gr \Lambda}(\bF,N)\to \Tor_i^{\gr \Lambda}(\bF,N/IN),$$
    for $i\geq 0$. 
\end{corollary}

\begin{corollary}
    The graded right $\gr \Lambda$-module $\Ext^{2f}_{\gr \Lambda}(N,\gr \Lambda)$ is supported in degrees $\leq 4f$, and $\bF\otimes_{\gr \Lambda}\Ext^{2f}_{\gr \Lambda}(N,\gr \Lambda)$ is supported in degrees $d\in[3f,4f]$.  
\end{corollary}

\begin{proof}
    See \cite[Cor.~2.3.4]{BHH$^+$25c}.
\end{proof}

\begin{lemma}\label{tor-of-N}
    For each $i\geq 0$ we have 
    $$\Tor_i^{\gr \Lambda}(\bF,N)\cong \bigoplus_{\chi\in W_D(\bar{\rho})}(\chi^{-1})^{\oplus m_i}$$
    as $H$-modules, where $m_i$ is defined in Condition~\ref{condition-iii}. 
    In particular, $\Tor_i^{\gr \Lambda}(\bF,N)\cong\Tor^\Lambda_i(\bF,\pi^\vee)$ as $H$-modules. 
    Moreover, $\Tor_i^{\gr \Lambda}(\bF,N)$ is supported in degrees $[-2i,-i]$. 
\end{lemma}

\begin{proof}
    See \cite[Lemma~2.3.5]{BHH$^+$25c}. 
\end{proof}

\subsection{A detailed study of quaternionic Serre weights}

First, we prove a result about the multiplicity of $N$, which requires the following key lemma. 

\begin{lemma}\label{key-lemma}
    Suppose that $\bar{\rho}$ is $(m+1)$-generic,  
    and suppose that $\chi,\chi'\in W_D(\bar{\rho})$ and that there exist integers $i_j$ with $|i_j|\leq m$ for all $j$ such that $\chi\prod_{j=0}^{f-1}\alpha_j^{i_j}=\chi'$. 
    Write $\fa(\chi)=(t_0,\ldots,t_{f-1})$. Then we have $|i_j|\leq 1$ for all $j$, and $i_j=-1$ (resp. $i_j=1$) implies $t_j=y_j$ (resp. $t_j=z_j$). 
\end{lemma}

\begin{proof}
    Write $\chi=\chi_{w,d}$, $\chi'=\chi_{w',d'}$. Then the equation $\chi\prod_{j=0}^{f-1}\alpha_j^{i_j}=\chi'$ becomes 
    $$\sum_{j=0}^{f-1}q^{w_j}p^jr_j+(1-q)\sum_{j=0}^{f-1}p^j(d_j-i_j)\equiv\sum_{j=0}^{f-1}q^{w'_j}p^jr_j+(1-q)\sum_{j=0}^{f-1}p^jd'_j\mod(q^2-1).$$
    That is, 
    $$\sum_{j=0}^{f-1}(q^{w_j}-q^{w'_j})p^jr_j- (1-q)\sum_{j=0}^{f-1}p^j(d'_j-d_j+i_j)\equiv 0\mod(q^2-1).$$
    Note that
    $$|\sum_{j=0}^{f-1}(q^{w_j}-q^{w'_j})p^jr_j- (1-q)\sum_{j=0}^{f-1}p^j(d'_j-d_j+i_j)|\leq (q-1)\sum_{j=0}^{f-1}p^j(r_j+m+2)<q^2-1,$$ 
    since $\bar{\rho}$ is $(m+1)$-generic. So we have 
    $$\sum_{j=0}^{f-1}(q^{w_j}-q^{w'_j})p^jr_j- (1-q)\sum_{j=0}^{f-1}p^j(d'_j-d_j+i_j)=0.$$
    Using that $q^{w_j}-q^{w'_j}=(q-1)(w_j-w_j')$, we have 
    $$\sum_{j=0}^{f-1}p^j((w_j-w'_j)r_j+d'_j-d_j+i_j)=0,$$
    and hence 
    $$r_j(w_j-w_j')+d_j'-d_j+i_j=0,$$
    for each $0\leq j\leq f-1$ since $\bar{\rho}$ is $(m+1)$-generic. 
    
    If $w_j\ne w'_j$, 
    we may assume that $w_{j}=1, w'_{j}=0$.
    Assume first that $j>0$. Then, by Relation~\ref{RI}, we have $d_j\in\{0,-1\}$ and $d_j'\in\{0,1\}$, 
    so $r_j=d_j-d_j'-i_j\leq m$, contradicting the assumption that $\bar{\rho}$ is $(m+1)$-generic. 
    Now suppose that $j=0$. If $\bar{\rho}$ is irreducible, 
    Relation~\ref{RII} again gives $d_j\in\{0,-1\}$ and $d_j'\in\{0,1\}$,
    yielding the same contradiction. 
    The case where $\bar{\rho}$ is reducible is similar. Thus, we conclude that we must have $w=w'$. 
    
    Now the equation becomes 
    $$\sum_{j=0}^{f-1}p^j(d'_j-d_j+i_j)=0.$$
    We conclude that $i_j=d_j-d'_j$. 
    When $w_j=0$, $d_{j}\in\{0,1\}$; when $w_j=1$, $d_j\in\{0,-1\}$. 
    So we always have $|i_j|\leq 1$. 

    To finish the proof, it suffices to prove the following lemma. 
    
    \begin{lemma}\label{convex-of-d}
    For a fixed $w$, let $S_w=S_w(\bar{\rho})$ be the set of all $d\in\{0,\pm1\}^f$ such that 
    $(w,d)$ corresponds to a $\chi\in W_{D}(\bar{\rho})$. 
    If $d,d'\in S_w$, then for any $f$-tuple $d''$ such that 
    $d''_j=d_j$ or $d_j'$, we have $d''\in S_w$. 
    \end{lemma}
    
    Suppose that Lemma~\ref{convex-of-d} holds. Then if $i_j=-1$, we have $d_j'=d_j+1$. 
    Define $d''$ by $d''_k=d_k+\delta_{j,k}$. Then $d''\in S_{w}$ by Lemma~\ref{convex-of-d}. 
    Note that $\chi_{w,d''}$ is exactly $\chi\alpha_j^{-1}$, so we have $t_j=y_j$. 
    Similarly $i_j=1$ implies $t_j=z_j$. 
\end{proof}

\begin{proof}[Proof of Lemma~\ref{convex-of-d}]
    First, consider the case where $\bar{\rho}$ is semisimple. 
    We need to show that $(w,d'')$ still satisfies Relation~\ref{RI} together with Relation~\ref{RII} or Relation~\ref{RII'}, 
    which is tautological. We check Relation~\ref{RI}, for example. 
    For $j>0$, note that Relation~\ref{RI} only involves $w_{j-1}$, $w_j$, and $d_j$. 
    Since both $(w_{j-1},w_j,d_j)$ and $(w_{j-1},w_j,d'_j)$ satisfy Relation~\ref{RI}, 
    and since $d''_j=d_j$ or $d_j'$, the tuple $(w_{j-1},w_j,d_j'')$ also satisfies this relation. 

    Now, assume that $\bar{\rho}$ is reducible non-split. 
    Recall that, in Section~\ref{quat-serre-weight}, we defined $V_{\chi}=\{v:\sigma_v(\bar{\rho}^{ss})\in\JH(\overline{\Theta([\chi])})\}$ and $v_\chi$ to be its minimal element. 
    We also showed that $W_D(\bar{\rho})=\{\chi\in W_D(\bar{\rho}^{ss}): v_{\chi}\leq v(\bar{\rho})\}$. 
    Now define the map 
    $$\Psi_w:S_w(\bar{\rho}^{ss})\to \{0,1\}^f,\quad d\mapsto v_{\chi},$$
    where $\chi=\chi_{w,d}$. 
    Then $S_w(\bar{\rho})=\{d\in S_w(\bar{\rho}^{ss}):\Psi_w(d)\leq v(\bar{\rho})\}$. 
    Hence it suffices to show that $\Psi_{w}(d'')\leq v(\bar{\rho})$. 

    From Equation~\ref{w_i} and Tables~\ref{table1} and~\ref{table2}, we can deduce the formula for $\Psi_{w}$. 
    We consider two cases to illustrate the general idea. Consider the case $j>0$ and $(w_{j-1},w_j)=(0,0)$. 
    Then we have $u_{j-1}+v_{j-1}\equiv u_j+v_j\equiv 0 \mod 2$, 
    so $(u_{j-1},u_j)=(v_{j-1},v_j)$, and we can read from Table~\ref{table1} that 
    $$d_j=\begin{cases}
        0, \quad (v_{j-1},v_j)=(0,0)\, or \,(0,1),\\
        1, \quad (v_{j-1},v_j)=(1,0)\, or\, (1,1).\\
    \end{cases}$$
    So we have $d_j=v_{j-1}$. That is, for any $v\in V_\chi$,  
    we must have $v_{j-1}=d_j$. 
    Hence $\Psi_{w}(d)_{j-1}=d_j$ by definition. Now consider another case $j>0$ and $(w_{j-1},w_j)=(0,1)$. 
    The same computation shows that $d_j$ is always $-1$. So for each $a\in \{0,1\}$, 
    we can find a $(u,v)$ such that $v_{j-1}=a$ and 
    all the relations given by Equation~\ref{w_i} and Tables~\ref{table1} and~\ref{table2} are satisfied (recall that $w$ and $d$ are fixed). 
    So we have $\Psi_{w}(d)_{j-1}=0$. The general formula is given as follows.

    For $j>0$:
    \begin{table}[H]
        \centering
        \caption{}
        \label{table3}
        \begin{tabular}{|c|c|c|c|c|}
            \hline
           $(w_{j-1}, w_j)$ &(0,0) & $(0,1)$ & $(1,0)$ & $(1,1)$ \\ \hline
            
           $d_j$ & $v_{j-1}$ & $-1$ & $1$ & $-v_{j-1}$\\ \hline
            
            $\Psi_{w}(d)_{j-1}$ & $d_j$ & $0$&$0$&$-d_j$\\ \hline
        \end{tabular}
    \end{table}

For $j=0$:
\begin{table}[H]
    \centering
    \caption{}
    \label{table4}
    \begin{tabular}{|c|c|c|c|c|}
        \hline
       $(w_{f-1}, w_0)$ &(0,0) & $(0,1)$ & $(1,0)$ & $(1,1)$ \\ \hline
        
       $d_0$ & $1$ & $-v_{f-1}$ & $v_{f-1}$ & $-1$ \\ \hline
        
        $\Psi_{w}(d)_{f-1}$ & $0$ &$-d_0$& $d_0$&$0$\\ \hline
    \end{tabular}
\end{table}

In any case, we have $\Psi_w(d'')_{j}\leq\max\{\Psi_w(d)_{j}, \Psi_w(d')_j\}\leq v(\bar{\rho})_j$, as desired. 
\end{proof}

\begin{remark}\label{t_j-well-defined}
    In the proof of Lemma~\ref{key-lemma}, we showed that it cannot happen that $\chi,\chi\alpha_j,\chi\alpha_j^{-1}\in W_D(\bar{\rho})$ simultaneously, 
    since, for a fixed $w$, either $d_j\in\{0,1\}$ or $d_j\in\{0,-1\}$. 
    So $t_j$ is well-defined.
\end{remark}

\begin{corollary}\label{key-lemma-inverse}
    For $\chi\in W_D(\bar{\rho})$ with $\fa(\chi)=(t_0,\ldots,t_{f-1})$, 
    let $S$ be a subset of $\{j:t_j\ne y_jz_j\}$. 
    Then $\chi':=\chi\prod_{j\in S}\chi_{t_j}^{-1}\in W_{D}(\bar{\rho})$, where $\chi_{t_j}$ is the eigencharacter of $t_j$. 
    Moreover, if we write $\fa(\chi')=(t'_0,\ldots,t'_{f-1})$, then $t'_j=\begin{cases}
        \frac{y_jz_j}{t_j}\quad &j\in S,\\
        t_j\quad &j\notin S.\\
    \end{cases}$
\end{corollary}

\begin{proof}
   It follows directly from Lemma~\ref{convex-of-d}. 
\end{proof}

\begin{corollary}\label{N-multi-free}
    Suppose that $n>1$ and $\bar{\rho}$ is $(2n-1)$-generic. For each character $\chi$ of $H$ such that $[N/I^{(n)}N:\chi]\ne 0$, we have $[N/I^{(n)}N:\chi]=r$. 
\end{corollary}

\begin{proof}
    It suffices to show that $N^{(1)}/I^{(n)}N^{(1)}$ is multiplicity-free, with $N^{(1)}=\bigoplus_{\chi\in W_{D}(\bar{\rho})}\chi^{-1}\otimes R/\fa(\chi)$. 
    Note that the characters of $H$ occurring in $\chi^{-1}\otimes R/(\fa(\chi)+I^{(n)})=\chi^{-1}\otimes (\bigotimes_{j=0}^{f-1}\bF[y_j,z_j]/(t_j,y_j^n,z_j^n)) $ are given by $\chi^{-1}\prod_{j=0}^{f-1}\alpha_j^{i_j}$ with $|i_j|\leq n-1$, and $i_j\leq 0$ if $t_j=y_j$ (resp. $i_j\geq 0$ if $t_j=z_j$). 
    Suppose that $N^{(1)}/I^{(n)}N^{(1)}$ is not multiplicity-free. Then there exist $\chi,\chi'$ and $i,l\in \{0,\pm1,\ldots,\pm(n-1)\}^f$ such that $\chi^{-1}\prod_{j=0}^{f-1}\alpha_j^{i_j}=\chi'^{-1}\prod_{j=0}^{f-1}\alpha_j^{l_j}$ with $(\chi,i)\ne(\chi',l)$. 
    In particular, $i\ne l$, so we may assume $l_{j_0}>i_{j_0}$ for some $j_0$. 
    Write $\fa(\chi)=(t_j)$ and $\fa(\chi')=(t_j')$, 
    and apply Lemma~\ref{key-lemma} to $\chi\prod_{j=0}^{f-1}\alpha_j^{l_j-i_j}=\chi'$. 
    We conclude that $l_{j_0}-i_{j_0}=1$ and $t_{j_0}=z_{j_0}$. 
    Interchanging the roles of $\chi$ and $\chi'$, we have $t'_{j_0}=y_{j_0}$. 
    But then $i_{j_0}\geq 0\geq l_{j_0}$, a contradiction. 
\end{proof}

Next we prove a result about essential self-duality. 
Note that, for $\chi\in W_D(\bar{\rho})$, we have $\chi^{q}\in W_D(\bar{\rho})$ since $\JH(\overline{\Theta([\chi])})=\JH(\overline{\Theta([\chi^q])})$.

\begin{lemma}\label{chi^q}
    Let $\chi\in W_D(\bar{\rho})$ with $\fa(\chi)=(t_0,\ldots,t_{f-1})$. 

    (1) Write $\fa(\chi^q)=(t_0^-,\ldots,t_{f-1}^-)$. 
    Then $t_j^-=\begin{cases}
        \frac{y_jz_j}{t_j}\quad & t_j\ne y_jz_j,\\
        y_jz_j\quad & t_j=y_jz_j.\\
    \end{cases}$
    
    (2) If $\chi$ corresponds to $(w,d)$, then
    $\chi^{q}$ corresponds to $(w^-, d^-)$ with $w^-_j=1-w_j, d^-_j=-d_j$. 
\end{lemma}

\begin{proof}
    (1) Note that if $\chi\alpha_j^{-1}\in W_D(\bar{\rho})$, then 
     $\chi^q\alpha_j=(\chi\alpha_j^{-1})^q\in W_D(\bar{\rho})$. 
     Similarly, $\chi\alpha_j\in W_D(\bar{\rho})$ implies $\chi^q\alpha_j^{-1}\in W_D(\bar{\rho})$. 

     (2) We have $\chi=\iota^{\sum_{j=0}^{f-1}q^{w_j}p^jr_j+(1-q)\sum_{j=0}^{f-1}p^jd_j}$. 
     Note that 
     $$\begin{aligned}
        q(\sum_{j=0}^{f-1}q^{w_j}p^jr_j+(1-q)\sum_{j=0}^{f-1}p^jd_j)&={\sum_{j=0}^{f-1}q^{1+w_j}p^jr_j+(1-q)q\sum_{j=0}^{f-1}p^jd_j}\\
        &\equiv \sum_{j=0}^{f-1}q^{1-w_j}p^jr_j+(1-q)\sum_{j=0}^{f-1}p^j(-d_j)\mod q^2-1. \\
     \end{aligned}$$
     so we have $\chi^q=\iota^{\sum_{j=0}^{f-1}q^{w_j^-}p^jr_j+(1-q)\sum_{j=0}^{f-1}p^jd_j^-}$, as desired. 
\end{proof}

\begin{definition}
    For $\chi=\chi_{w,d}\in W_D(\bar{\rho})$ with $\fa(\chi)=(t_0,\ldots,t_{f-1})$, let $S=\{j:t_j\ne y_jz_j\}$ and 
    define $\eta_{\chi}=\prod_{j\in S}\chi_{t_j}$. 
    Moreover, let $\chi'=\chi\eta_{\chi}^{-1}$ and define $\chi^*=(\chi')^q$. 
\end{definition}

Write $\chi'=\chi_{w',d'}$ and $\chi^*=\chi_{w^*,d^*}$. We know that $w'=w$ by the proof of Lemma~\ref{key-lemma}, 
and we have $w^*=w^{-}$ and $d^*=(d')^-$. 

\begin{lemma}
    (1) We have $\fa(\chi^*)=\fa(\chi)$, and $\chi^{**}=\chi$. In particular, $\chi\mapsto\chi^*$ is a bijection on $W_D(\bar{\rho})$. 

    (2) We have $\chi\chi^*=\eta_\chi\cdot(\eta_c\circ \Nm_{\bF_{q^2}/\bF_q})$, where $\eta_c$ is the central character defined in Remark~\ref{cent-char}.
\end{lemma}

\begin{proof}
    (1) Write $\fa(\chi')=(t'_0,\ldots,t'_{f-1})$ and $\fa(\chi^*)=(t_0^*,\ldots,t_{f-1}^*)$. 
    For $j\in S=\{j:t_j\ne y_jz_j\}$, we have $t'_j=\frac{y_jz_j}{t_j}$ by Corollary~\ref{key-lemma-inverse}, 
    and $t_j^*=\frac{y_jz_j}{t'_j}$ by Lemma~\ref{chi^q}(1). 
    For $j\notin S$, we have $t'_j=t_j$ and hence $t_j^*=t_j$ by Lemma~\ref{chi^q}(1). So in all cases $t_j^*=t_j$. 

    If we write $\eta_{\chi}=\prod_{j\in S}\chi_{t_j}$, then 
    $\eta_{\chi^*}=\eta_\chi$ by the argument above. Note that $\eta_\chi^q=\eta_\chi^{-1}$ since $\alpha_j^q=\alpha_j^{-1}$ for each $j$. 
    So we have  
    $$\chi^{**}=(\chi^*\eta_{\chi^*}^{-1})^q=(\chi^q\eta_\chi\eta_\chi^{-1})^q=\chi^{q^2}=\chi.$$

    (2) Note that 
    $\chi\chi^*\eta_\chi^{-1}=\chi\chi^{q}=\iota^{\sum_{j=0}^{f-1}(q+1)p^jr_j}=\eta_c\circ \Nm_{\bF_{q^2}/\bF_q}$. 
\end{proof}

\begin{corollary}\label{ess-self-dual}
    We have an isomorphism 
    $$\Ext_{\gr \Lambda}^{2f}(N, \gr \Lambda)\cong N\otimes (\eta_c\circ \Nm_{\bF_{q^2}/\bF_q})$$
    as $\gr \Lambda$-modules with compatible $H$-action. 
\end{corollary}

\begin{proof}
    Recall that $\overline{R}=R/(y_0z_0,\ldots,y_{f-1}z_{f-1})$. 
    Since $N$ is an $\overline{R}$-module, 
    by a standard d\'evissage argument as in \cite[Lemma~3.3.1.9]{BHH$^+$25a}, 
    we have 
    $$\Ext_{\gr \Lambda}^{2f}(N, \gr \Lambda)=\Ext_{R}^{f}(N, R)=\Hom_{\overline{R}}(N,\overline{R}).$$
    By the same argument as in \cite[3.3.1.10]{BHH$^+$25a}, we have 
    $$\Hom_{\overline{R}}(R/\fa(\chi),\overline{R})\cong \eta_{\chi}^{-1}\otimes R/\fa(\chi).$$
    So 
    $$\begin{aligned}
        \Hom_{\overline{R}}(N,\overline{R})&=(\bigoplus_{\chi\in W_D(\bar{\rho})}(\chi\eta_{\chi}^{-1}\otimes R/\fa(\chi)))^{\oplus r}\\
        &=(\bigoplus_{\chi\in W_D(\bar{\rho})}(\eta_c\circ \Nm_{\bF_{q^2}/\bF_q})(\chi^*)^{-1}\otimes R/\fa(\chi^*))^{\oplus r}\\
        &=N\otimes (\eta_c\circ \Nm_{\bF_{q^2}/\bF_q}),
    \end{aligned}
    $$
    as desired. 
\end{proof}

\subsection{The representation \texorpdfstring{$\tau$}{tau} and a \texorpdfstring{$\Tor$}{Tor}-injective property}

We first extend some results of \cite[\S\S~2.1, 2.2]{HW24b} concerning uniserial $\cO_D^\times$-representations. 
Let $\chi: \cO_D^\times\to \bF^\times$ be a smooth character. We study the induced representation $\Indu{U_D^2H}{\cO_D^\times}{\chi}$. 
Note that the quotient $U_D^2H\backslash \cO_D^\times$ has a set of representatives 
$$\{1+\varpi_{D}[\lambda]:\lambda\in \bF_{q^2}\}.$$
Thus, by fixing a basis vector $v$ of $\chi$, we obtain a basis of $\Indu{U_D^2H}{\cO_D^\times}{\chi}$ as follows: 
$$\{[1+\varpi_D[\lambda],v]:\lambda\in\bF_{q^2}\}.$$
For $0\leq k\leq q^2-1$, define 
$$f_{k,v}=\sum_{\lambda\in \bF_{q^2}}\lambda^k[1+\varpi_D[\lambda],v].$$
These vectors also form a basis of $\Indu{U_D^2H}{\cO_D^\times}{\chi}$. Moreover, 
$$[\mu](1+\varpi_D[\lambda])=(1+\varpi_D[\mu^{q-1}\lambda])[\mu].$$
So 
$$[\mu]f_{k,v}=\sum_{\lambda\in \bF_{q^2}}\lambda^k[1+\varpi_D[\mu^{q-1}\lambda],\chi(\mu)v]=\chi(\mu)\mu^{(1-q)k}f_{k,v},$$
and hence $f_{k,v}$ is an $H$-eigenvector of character $\chi\iota^{(1-q)k}$. 

For $0\leq k\leq q^2-1$, write $k=\sum_{j=0}^{2f-1}p^jk_j$. 
We say that $k\preccurlyeq k'$ if $k_j\leq k'_j$ for each $0\leq j\leq 2f-1$, and 
$k\prec k'$ if $k\preccurlyeq k'$ but $k\ne k'$.  

\begin{lemma}
    The subrepresentation of $\Indu{U_D^2H}{\cO_D^\times}{\chi}$ generated by $f_{k,v}$ is spanned by 
    $\{f_{k',v}:k'\preccurlyeq k\}$. Moreover, the $\cO_D^\times$-socle of $\Indu{U_D^2H}{\cO_D^\times}{\chi}$ is $\bF f_{0,v}$. 
\end{lemma}

\begin{proof}
    The method is the same as that in \cite[Prop.~2.7]{HW24b}. 
    Since $f_{k,v}$ is an $H$-eigenvector and $U^2_D$ acts trivially, it suffices to consider the action of $U_D^1/U_D^2$. 
    We have 
    $$(1+\varpi_D[\mu])f_{k,v}=\sum_{\lambda\in\bF_{q^2}}\lambda^k[(1+\varpi_D[\mu+\lambda]),v]=\sum_{\lambda\in\bF_{q^2}}(\lambda-\mu)^k[(1+\varpi_D[\lambda]),v].$$
    Using the formula
    $$(\lambda-\mu)^k=\sum_{k'\preccurlyeq k}\prod_{j=0}^{2f-1}\binom{k_j}{k'_j}(-\mu)^{k-k'}\lambda^{k'}$$
    in $\bF_{q^2}$, we get 
    $$(1+\varpi_D[\mu])f_{k,v}=\sum_{k'\preccurlyeq k}\prod_{j=0}^{2f-1}\binom{k_j}{k'_j}(-\mu)^{k-k'} f_{k',v}.$$
    So we conclude that $\langle\cO_D^\times\cdot f_{k,v}\rangle\subset \bigoplus_{k'\preccurlyeq k}\bF f_{k',v}$. 
    On the other hand, for fixed $k'\prec k$, we have 
    $$\sum_{\mu\in {\bF_{q^2}}}\mu^{q^2-1-(k-k')}(1+\varpi_D[\mu])f_{k,v}=(-1)^{k-k'+1}\prod_{j=0}^{2f-1}\binom{k_j}{k'_j}f_{k',v}.$$
    We conclude that $f_{k',v}\in \langle\cO_D^\times\cdot f_{k,v}\rangle$. 

    For the last assertion, let $f\in (\Indu{U_D^2H}{\cO_D^\times}{\chi})^{U_D^1}$ be a nonzero element. 
    Write $f=\sum_{k=0}^{q^2-1} a_kf_{k,v}$ and let $l$ be a maximal index with respect to $\preccurlyeq$ such that $a_l\ne 0$. 
    We want to show that $l=0$. Note that for each $k\preccurlyeq l$, 
    we have 
    $$\sum_{\mu\in {\bF_{q^2}}}\mu^{q^2-1-l}(1+\varpi_D[\mu])f_{k,v}=\begin{cases}
        (-1)^{l+1}f_{0,v}, & l=k.\\
        0, & l\ne k.
    \end{cases} $$
    So $\sum_{\mu\in {\bF_{q^2}}}\mu^{q^2-1-l}(1+\varpi_D[\mu]) f\ne 0$. 
    But if $l\ne 0$, then $\sum_{\mu\in {\bF_{q^2}}}\mu^{q^2-1-l}(1+\varpi_D[\mu])\in\fm$, 
    and hence kills $f$, which is a contradiction. 
\end{proof}

\begin{proposition}\label{ext1-of-character}
    Let $\chi,\psi:\cO_D^\times\to \bF^\times$ be two smooth characters. 
    Then $\Ext^1_{\cO_D^\times/Z_D^1}(\psi,\chi)\ne 0$ if and only if $\psi=\chi\alpha_j$ or $\psi=\chi\alpha_j^{-1}$, 
    and in either case $\dim_{\bF}\Ext^1_{\cO_D^\times/Z_D^1}(\psi,\chi)=1$. 
\end{proposition}

\begin{proof}
    This is a corollary of \cite[Cor.~2.11]{HW24a}.
\end{proof}

Let $E_j^-(\chi)$ be the unique non-split extension 
\begin{equation}
    0\to\chi\to E_j^-(\chi)\to\chi\alpha_j^{-1}\to 0,
\end{equation}
and similarly let $E^+_j(\chi)$ be the unique non-split extension 
\begin{equation}
    0\to\chi\to E_j^+(\chi)\to\chi\alpha_j\to 0.
\end{equation}

\begin{lemma}
    Suppose that $0\leq s\leq p-1$. There exists a unique smooth representation of $\cO_D^\times/Z_D^1$ over $\bF$, 
    denoted by $E_j^-(\chi,s)$, which is uniserial of dimension $s+1$, satisfies $\dim_{\bF} E^-_j(\chi,s)[\fm^i]=i$ for $0\leq i\leq s+1$, and has graded pieces $\chi,\chi\alpha_j^{-1},\ldots,\chi\alpha_j^{-s}$. 
\end{lemma}

\begin{proof}
    Existence follows by taking the subrepresentation of $\Indu{U_D^2H}{\cO_D^\times}{\chi}$ generated by $f_{p^js,v}$. 
    For the uniqueness, it suffices to show that $\dim_{\bF}\Ext^1_{\cO_D^\times/Z_D^1}(\chi\alpha_j^{-(i+1)},E_j^-(\chi,i))\leq 1$. 
    Using the exact sequence 
    $$0\to E^-_j(\chi,i-1)\to E^-_j(\chi,i)\to \chi\alpha_j^{-i}\to 0,$$
    we have 
    $$\Ext^1(\chi\alpha_j^{-(i+1)},E^-_j(\chi,i-1))\to \Ext^{1}(\chi\alpha_j^{-(i+1)},E^-_j(\chi,i))\to\Ext^1(\chi\alpha_j^{-(i+1)},\chi\alpha_j^{-i}).$$
    By Proposition~\ref{ext1-of-character}, the first term is zero and the third term has dimension $1$, so the result follows. 
\end{proof}

\begin{remark}
    \begin{sloppypar}
    In \cite[Lemma~2.3]{HW24b}, the case $f=1$ is treated by using the fact that 
    $\Ext^2_{\cO_D^\times/Z_D^1}(\psi,\chi)\ne 0$ if and only if $\psi=\chi\alpha_j$ or $\psi=\chi\alpha_j^{-1}$, 
    which may not be true for general $f$. 
    \end{sloppypar}
\end{remark}

Similarly, there is a uniserial representation $E_j^+(\chi,s)$ with graded pieces $\chi,\chi\alpha_j,\ldots,\chi\alpha_j^{s}$. 
Once we have constructed these uniserial representations, we can construct the representation $\tau$ analogous to the one in 
\cite[\S~2.4]{BHH$^+$25c}.

\begin{lemma}
    Suppose that $1\leq n\leq p$. There exists a finite-dimensional smooth representation $\tau^{(n)}$ of $\cO_D^{\times}$ over $\bF$ such that 
    $$\gr_\fm((\tau^{(n)})^\vee)\cong N/I^{(n)}N$$
    as graded $\gr \Lambda$-modules with compatible $H$-action. 
    More precisely, $\tau^{(n)}\cong (\bigoplus_{\chi\in W_D(\bar{\rho})}\tau_{\chi}^{(n)})^{\oplus r}$ with 
    $$\gr_\fm((\tau_\chi^{(n)})^\vee)\cong \chi^{-1}\otimes R/(I^{(n)}+\fa(\chi))$$
    as graded  $\gr \Lambda$-modules with compatible $H$-action. 
    In particular, $\tau_{\chi}^{(n)}[\fm]\cong \chi$. 
\end{lemma}

\begin{proof}
    This is the same as \cite[Lemma~2.4.1]{BHH$^+$25c}. We recall the construction for convenience of the reader. 
    Write $E_j^-(s)=E_j^-(1,s)$ and $E_j^+(s)=E_j^+(1,s)$ for $0\leq s\leq p-1$. 
    A direct calculation shows that
    $$\gr_\fm(E_j^-(s)^\vee)\cong \bF[y_j,z_j]/(y_j^{s+1}, z_j), \quad \gr_\fm(E_j^+(s)^\vee)\cong \bF[y_j,z_j]/(y_j, z_j^{s+1}). $$
    Moreover, for the amalgamated sum $E_j^-(s)\oplus_{1}E_j^+(s):=(E_j^-(s)\oplus E_j^+(s))/1$, we have 
    $$\gr_\fm\big((E_j^-(s)\oplus_{1}E_j^+(s))^\vee\big)\cong \bF[y_j,z_j]/(y_j^{s+1}, y_jz_j, z_j^{s+1}). $$
    Define 
    $$W_{\chi,j}=\begin{cases}
        E_j^{+}(n-1) \quad & t_j=y_j,\\
        E_j^{-}(n-1) \quad & t_j=z_j,\\
        E_j^-(n-1)\oplus_1 E_j^{+}(n-1) \quad & t_j=y_jz_j,\\
    \end{cases}$$
    and $\tau_\chi^{(n)}=\chi\otimes (\bigotimes_{j=0}^{f-1}W_{\chi,j})$. 
    For simplicity we write $M_j=W_{\chi,j}^\vee$ and $M=\bigotimes_{j=0}^{f-1} M_j$. 
    In all cases we have 
    $$\gr_\fm M_j=\bF[y_j,z_j]/(y_j^n, z_j^n,t_j).$$
    Denote by $C_\bullet$ the tensor product filtration on $M$. 
    Then $\gr_{C_\bullet}(M)\cong \bigotimes_{j=0}^{f-1}\gr_\fm(M_j)\cong R/(I^{(n)}+\fa(\chi))$. 
    By \cite[Lemma~1.1(i)]{AJL83}, we have $\fm^d M\subset C_{-d}(M)$, so we have a natural graded morphism 
    $$\phi:\gr_\fm M\to \gr_{C_\bullet}(M).$$
    To show $\phi$ is an isomorphism, it suffices to show $\phi$ is surjective as $M$ is finite-dimensional. 
    We have $\fm^0M=C_0(M)=M$, so $\phi$ is surjective on the degree-$0$ part. 
    Since $\gr_{C_\bullet}(M)$ is generated by its degree-$0$ part, $\phi$ is surjective by Nakayama's lemma. 
\end{proof}

Since $\gr_\fm\tau^\vee$ has the same form as its analogue in \cite[\S~2.4]{BHH$^+$25c}, the arguments there also prove the following results.

\begin{lemma}\label{(2n-1)-generic}
    Suppose that $\bar{\rho}$ is $(2n-1)$-generic. Then there exists an $\cO_D^\times$-equivariant embedding $\tau^{(n)}\hookrightarrow \pi|_{\cO_D^\times}$ such that the composite of the induced maps $N\to \gr_\fm\pi^\vee\to \gr_\fm(\tau^{(n)})^\vee$ is the natural map $N\to N/I^{(n)}N$. 
    In particular, the composite map is an isomorphism in degree $\geq -(n-1)$ and $\tau^{(n)}[\fm^n]=\pi[\fm^n]$. 
\end{lemma}

\begin{proof}
    See \cite[Lemma~2.4.2]{BHH$^+$25c}.
\end{proof}

\begin{corollary}
    Suppose that $\bar{\rho}$ is $(2n-1)$-generic. Then $\bigoplus_{\chi\in W_D(\bar{\rho})}\tau_\chi^{(n)}$ is multiplicity-free, and every irreducible constituent of $\pi[\fm^n]=\tau^{(n)}[\fm^n]$ occurs with multiplicity $r$. 
\end{corollary}

\begin{proof}
    See \cite[Cor.~2.4.3]{BHH$^+$25c}.
\end{proof}

Let $\tau=\tau^{(3)}$, so $\gr_\fm\tau^\vee\cong N/IN$, 
which admits a minimal gr-free resolution $G_\bullet=G'_\bullet\oplus G''_\bullet$, 
with $G'_\bullet$ a minimal gr-free resolution of $N$. 
More precisely, let $G_{\chi,\bullet}$ be the minimal gr-free resolution of $\chi^{-1}\otimes R/(I+\fa(\chi))$ and write $G_\bullet=(\bigoplus_{\chi}G_{\chi,\bullet})^{\oplus r}$. 
By Proposition~\ref{properties-of-resolution}(4), for each $\chi$ we can lift $G_{\chi,\bullet}$ to a strict filt-free resolution $L_{\chi,\bullet}$ of $\tau_\chi^\vee$. 
We can give a compatible $H$-action on each $L_{\chi,\bullet}$ as in Proposition~\ref{properties-of-resolution}(2). 
Let $L_\bullet=(\bigoplus_{\chi}L_{\chi,\bullet})^{\oplus r}$; then $L_\bullet$ is a strict filt-free resolution of $\tau^\vee$ with compatible $H$-action. 

\begin{lemma}
    For $i\geq 0$, there exists a decomposition $L_i=L_i'\oplus L_i''$ as filt-free $\Lambda$-modules with compatible $H$-action 
    whose associated graded decomposition is $G_i=G'_i\oplus G_i''$. 
\end{lemma}

\begin{proof}
    See \cite[Lemma~2.4.5]{BHH$^+$25c}, whose proof applies to general filtered modules. 
\end{proof}

\begin{lemma}\label{min-of-L}
    Suppose that $\bar{\rho}$ is $5$-generic. Then $L_\bullet$ is minimal. 
    Moreover, for $i\in\{0,1,2\}$, $L_i=L_i'\oplus L''_i$ satisfies conditions (1) and (2) of Lemma~\ref{lemma-of-filt}. 
\end{lemma}

\begin{proof}
    See \cite[Lemma~2.4.6]{BHH$^+$25c}.
\end{proof}

\begin{corollary}\label{degenerated-spec-seq}
    Suppose that $\bar{\rho}$ is $5$-generic. For $i\geq 0$, we have a canonical isomorphism
    $$\Tor_i^{\gr \Lambda}(\bF, \gr_\fm\tau^\vee)\cong \gr\Tor_i^\Lambda(\bF, \tau^\vee).$$
\end{corollary}

\begin{proof}
    See \cite[2.4.8]{BHH$^+$25c}. 
\end{proof}

\begin{proposition}\label{tor-inj}
    Assume that $\bar{\rho}$ is $5$-generic. For $i\in\{0,1,2\}$, the natural morphism
    $$\varphi_i: \Tor_i^\Lambda(\bF, \pi^\vee)\to \Tor_i^\Lambda(\bF, \tau^\vee)$$
    is injective. 
\end{proposition}

\begin{proof}
    This is essentially the same as \cite[Prop.~2.4.9]{BHH$^+$25c}. 
    We weaken the genericity hypothesis from $9$-genericity to $5$-genericity through a more detailed analysis of a specific spectral sequence. 
    Recall that it suffices to prove the following: there exist separated filtrations on 
    $\Tor_i^\Lambda(\bF,\pi^\vee)$ and $\Tor_i^\Lambda(\bF,\tau^\vee)$, such that $\varphi_i$ is a filtered morphism and 
    the induced $\gr(\varphi_i)$ is injective. 
    Choose a minimal resolution of $\gr_\fm\pi^\vee$ and lift it to a filt-free resolution $M_\bullet$ of $\pi^\vee$.
    This produces a spectral sequence $\{E^r_i:r\geq 0,i\geq 0\}$ with 
    the following properties (see \cite[\S~III.1]{LvO96} for details): 

    (1) Each $E_i^r=\bigoplus_{n\in\bZ} E_{i,n}^r$ is a graded vector space with an $H$-action, 
    and $E_i^0=\gr (\bF\otimes_\Lambda M_i)=\bF\otimes_{\gr \Lambda}\gr M_i$, $E_i^1=\Tor_i^{\gr \Lambda}(\bF,\gr_\fm\pi^\vee)$. 

    (2) For each $r\geq 1$, there exists a complex 
    $$\cdots\to E_1^r\to E_0^r\to 0,$$
    whose homology is $E_i^{r+1}$, and each differential map is a graded morphism of degree $-r$ and is $H$-equivariant. 

    (3) For each $i\geq 0$, 
    $E_i^r\cong E_i^\infty=\gr\Tor_i^\Lambda(\bF,\pi^\vee)$ for $r$ large enough. 
    
    Similarly, replacing $\pi^\vee$ by $\tau^\vee$ and using its minimal filt-free resolution $L_\bullet$, 
    we get another spectral sequence $\{E^{\prime\,r}_i:r\geq 0,i\geq 0\}$ with similar properties. 
    Moreover, by a standard argument in homological algebra, 
    the morphism $\pi^\vee \twoheadrightarrow \tau^\vee$ extends to a filtered morphism of complexes $M_\bullet\to L_\bullet$ with compatible $H$-actions. 
    Hence we get a morphism of spectral sequences: 
    $$\begin{array}{ccc}
E_i^r & \Longrightarrow &
\Tor_i^{\Lambda}(\bF,\pi^\vee)\\[3pt]
\big\downarrow && \big\downarrow\rlap{\(\scriptstyle\varphi_i\)}\\[3pt]
E_i^{\prime\,r} & \Longrightarrow &
\Tor_i^{\Lambda}(\bF,\tau^\vee).
\end{array}$$
By Corollary~\ref{degenerated-spec-seq}, 
the bottom spectral sequence degenerates on the $r=1$ page. 
Now it suffices to show that 
    $\gr \varphi_i: E_i^\infty=\gr\Tor_i^\Lambda(\bF,\pi^\vee)\to \gr\Tor_i^\Lambda(\bF,\tau^\vee)=E_i^{\prime\,\infty}$
is injective for $i=0,1,2$. 

The cases $i=0,1$ are the same as in \cite[Prop.~2.4.9]{BHH$^+$25c}, and $5$-genericity is enough here. 
So we assume that $i=2$. In the cited proof, the condition that $\bar{\rho}$ is $9$-generic is used only 
to show that $N\twoheadrightarrow \gr_\fm\pi^\vee$ is an isomorphism in degree $\geq -4$ by Lemma~\ref{(2n-1)-generic}, 
from which the authors deduce that the morphism $E_1^1\to E^\infty_1$ is an isomorphism in degree $\geq -4$. 
We will show that under the assumption that $\bar{\rho}$ is $5$-generic, $E_1^1\to E^\infty_1$ is still an isomorphism in degree $\geq -4$. 
Note that this is equivalent to saying that $d^r:E_2^r\to E_1^r$ is zero in degree $\geq -4$ for each $r\geq 1$. 

As in the proof of \cite[Prop.~2.4.9]{BHH$^+$25c}, we can deduce that $N\twoheadrightarrow \gr_\fm\pi^\vee$ is an isomorphism in degree $\geq -2$ by Lemma~\ref{(2n-1)-generic}, 
and that $d^r:E_2^r\to E_1^r$ is zero in degree $\geq -2$. 
So we need only to deal with degrees $-3$ and $-4$. 
Note that $d^r$ takes the form 
$$d^r_n: E_{2,n}^r\to E^r_{1,n-r}$$
on each graded piece. So it suffices to show 
that for each $r\geq 1$, the maps 
$d^r_{-3+r}: E^r_{2,-3+r}\to E^r_{1,-3}$ and 
$d^r_{-4+r}: E^r_{2,-4+r}\to E^r_{1,-4}$ are zero. 

We need the following lemma, which is a variation of Lemma~\ref{N-multi-free} and can be proved in the same way. We postpone its proof until after the proof of the proposition.

\begin{lemma}\label{char-sep}
    Let $\cX=\{\chi^{-1}:\chi\in W_D(\bar{\rho})\}$ and 
    let $N_s$ be the degree-$s$ part of $N$. 
    Assume that \(\bar{\rho}\) is \(5\)-generic. Then for \(1\leq s\leq4\), we have 
        \[
          \JH_H(N_{-s})\cap\cX=\varnothing.
        \]   
\end{lemma}

We first control the targets of these differential maps.
Let $V_i^\prime=\Tor_i^{\gr \Lambda}(\bF,N)$ for each $i\geq 0$, 
and let $\cK=\ker(N\twoheadrightarrow \gr_\fm\pi^\vee)$. 
Then $\cK_n=0$ for $n\geq -2$ since the map is an isomorphism in degree $\geq -2$. 
From the short exact sequence 
\[
  0\longrightarrow \cK\longrightarrow N
  \longrightarrow\gr_\fm\pi^\vee
  \longrightarrow0,
\]
we get two exact sequences
\begin{equation}\label{eq1}
    V_1'\to E_1^1
    \to \bF\otimes_{\gr \Lambda}\cK
    \to \bF\otimes_{\gr \Lambda}N,
\end{equation}
and
\begin{equation}\label{eq}
    V_2'\to E_2^1
    \to \Tor_1^{\gr \Lambda}(\bF,\cK)
    \to V_1'.
\end{equation}
The module \(V_1'\) is supported in degrees \([-2,-1]\) by Lemma~\ref{tor-of-N}, while
\(\bF\otimes_{\gr \Lambda} N\) is supported in degree $0$. 
Hence, by Equation~\ref{eq1}, for \(n=-3,-4\) we have 
\begin{equation*}
  E^1_{1,n}\cong(\bF\otimes_{\gr \Lambda} \cK)_n,
\end{equation*}
which is a subquotient of \(N_n\), so we get $\JH_H(E^1_{1,-3})\subset \JH_H(N_{-3})$ 
and $\JH_H(E^1_{1,-4})\subset \JH_H(N_{-4})$.

We then control the sources of these differential maps. 
From $$0\to\fm_{\gr \Lambda}\to \gr \Lambda\to \bF\to 0,$$ 
we get 
\begin{equation*}
    0\to \Tor_1^{\gr \Lambda}(\bF,\cK)
    \to \fm_{\gr \Lambda}\otimes_{\gr \Lambda}\cK.
\end{equation*}
Since $\fm_{\gr \Lambda}$ is supported in degree $\leq -1$
and $\cK$ in degree $\leq -3$, 
we have $\Tor_1^{\gr \Lambda}(\bF,\cK)_{-3}=0$. 
Together with Equation~\ref{eq}, we get a surjection 
$$(V_2')_{-3}\twoheadrightarrow E_{2,-3}^1.$$
Then we get $\JH_{H}(E_{2,-3}^1)\subset \JH_H(V_2')_{-3}\subset \cX$ by Lemma~\ref{tor-of-N}. 
Applying Lemma~\ref{truncation} to $N\twoheadrightarrow \gr_\fm\pi^\vee$, 
we obtain an isomorphism $(V'_2)_{-2}\cong (E_2^1)_{-2}$. 
Since \(V_2'\) is supported in degrees \([-4,-2]\) by
Lemma~\ref{tor-of-N}, we have 
$E^1_{2,-2}\cong (V'_2)_{-2}$ and $E^1_{2,n}=0$ for $n\geq -1$. 
In conclusion, $\JH_H(E^r_{2,-4+r})\subset \JH_H(E^1_{2,-4+r})\subset \cX$ for each $r\geq 1$. 

Now we can finish our proof. Since $\JH_H(E^1_{1,-3})\subset \JH_H(N_{-3})$ 
and $\JH_H(E^1_{1,-4})\subset \JH_H(N_{-4})$, we deduce by Lemma~\ref{char-sep} that 
all the differentials $d^r_{-3+r}$ and $d^r_{-4+r}$ must be zero. Then 
$E_1^1\to E^\infty_1$ is still an isomorphism in degree $\geq -4$, 
and we can use the same argument as in \cite[Prop.~2.4.9]{BHH$^+$25c} to obtain the result. 
\end{proof}

\begin{proof}[Proof of Lemma~\ref{char-sep}]
    Suppose to the contrary. Then there exist $\chi, \chi'\in W_{D}(\bar{\rho})$ such that 
    $$\chi^{-1}\prod_{j=0}^{f-1}\alpha_j^{i_j}=\chi'^{-1},$$
    with each $|i_j|\leq s\leq 4$. So we can apply Lemma~\ref{key-lemma} and argue as in the proof of Lemma~\ref{N-multi-free} to derive a contradiction. 
\end{proof}

\subsection{Proof of the theorem}

We now prove our main theorem. Some results about characteristic cycles are needed. 
Recall that $\overline{R}=R/(y_jz_j:0\leq j\leq f-1)$, 
and the minimal prime ideals of $\overline{R}$ are given by $(y_i,z_j: i\in J, j\notin J)$ for some $J\subset \{0,\ldots,f-1\}$. 
Let $N$ be a finitely generated $\gr \Lambda$-module killed by $I_D$, so $N$ is naturally an $\overline{R}$-module. 
For any minimal prime ideal $\fq$ of $\overline{R}$, let $m_{\fq}(N)$ be the length of $N_{\fq}$ over $\overline{R}_{\fq}$. 
More generally, if $N$ is a finitely generated $\gr \Lambda$-module killed by some power of $I_D$, then for any minimal prime ideal $\fq$ of $\overline{R}$ we define 
$$m_{\fq}(N)=\sum_{n\geq 0} m_{\fq}(I_D^nN/I_D^{n+1}N),$$
which is a finite sum. 

\begin{definition}
    Let $N$ be a finitely generated $\gr \Lambda$-module killed by some power of $I_D$. 
    We define the characteristic cycle of $N$ by 
    $$Z(N)=\sum_{\fq}m_\fq(N)\fq,$$
    where $\fq$ runs over all minimal prime ideals of $\overline{R}$. 
\end{definition}

\begin{proposition}
    If $0\to N'\to N\to N''\to 0$ is exact, then $Z(N)=Z(N')+Z(N'')$. 
\end{proposition}

\begin{proof}
    It suffices to check that $m_{\fq}(N)=m_{\fq}(N')+m_{\fq}(N'')$, 
    which is proved in the same way as \cite[3.1.4.3]{BHH$^+$25a}. 
\end{proof}

Let $M$ be a finitely generated $\Lambda$-module with a good filtration $F$, 
such that $\gr_F M$ is killed by some power of $I_D$. 
Let $F'$ be another good filtration of $M$. 
Then by \cite[Lemma~I.5.3]{LvO96}, any two good filtrations of $M$ are equivalent, in the sense that there exists $c\in \mathbb{Z}$ such that 
$$F_{n-c}M\subset F'_nM\subset F_{n+c}M,$$ 
for any $n\in\mathbb{Z}$. So $\gr_{F'}(M)$ is also killed by some power of $I_D$. 
The following result allows us to compare the characteristic cycles arising from different filtrations. 

\begin{lemma}\label{cycle}
    Keeping $M$, $F$ and $F'$ as above, we have
    $$Z(\gr_F M)=Z(\gr_{F'}M).$$
\end{lemma}

\begin{proof}
    The proof is the same as that of \cite[3.3.4.3]{BHH$^+$25a}.
\end{proof}

\begin{proof}[Proof of Theorem~\ref{main-thm}]
    The proof is almost identical to that of \cite[Thm.~2.1.2]{BHH$^+$25c}. 
    We first show that $N$ is Cohen--Macaulay of grade $2f$ and essentially self-dual. 
    In Lemma~\ref{resolution-of-N}, we have constructed a free resolution of $N$ of length $2f$, 
    so $\Ext^i_{\gr \Lambda}(N,\gr \Lambda)=0$ for $i>2f$. 
    By \cite[Thm.~4.3]{Lev92}, if $M$ is a finitely generated module over an Auslander-Gorenstein ring $R$ and 
    $f:M\to M$ is an injective morphism of $R$-modules, then $j_R(M/fM)\geq j_R(M)+1$, 
    where $j_R(M)=\min\{i:\Ext_R^i(M,R)\ne 0\}$ is the grade of $M$. 
    Applying this result to $M=\gr\Lambda$ with the central regular sequence $h_0,\ldots,h_{f-1}$, and then $t_0,\ldots,t_{f-1}$, 
    we get that $j_{\gr \Lambda}(N)\geq 2f$, that is, $\Ext^i_{\gr \Lambda}(N,\gr \Lambda)=0$ for $i<2f$. 
    Finally, the essential self-duality follows from Corollary~\ref{ess-self-dual}.  

    Now we want to show that the surjection $N\to\gr_\fm\pi^\vee$ is also injective. 
    Denote its kernel by $N'$. 
    Since $N$ is self-dual, it is pure by \cite[Prop.~III.4.2.8(1)]{LvO96}. 
    So any nonzero submodule of $N$ has grade $2f$ over $\gr \Lambda$, and hence has grade $0$ over $\overline{R}$ by a d\'evissage argument as in \cite[Lemma~3.3.1.9]{BHH$^+$25a}. 
    In particular, it has a nonzero cycle. Thus, to show $N'=0$, 
    it suffices to show that $Z(N)=Z(\gr_\fm\pi^\vee)$.
    
    Let $P_\bullet$ be a minimal free resolution of $\pi^\vee$ with compatible $H$-action. 
    We claim that it suffices to equip each $P_i$, for $i\in\{0,1,2\}$, with a good filtration $F_i$ 
    such that $P_2\to P_1\to P_0$ is a complex of filtered $\Lambda$-modules, $H_1(\gr P_\bullet)=0$, and $H_0(\gr P_\bullet)\cong N$. 
    Indeed, since $\gr(P_2)\to\gr(P_1)\to\gr(P_0)$ is exact, 
    the map $P_1\to P_0$ is strict by \cite[Thm.~I.4.2.4(2)]{LvO96}. 
    Then $P_1\to P_0\to \pi^\vee\to 0$ is a strict exact sequence if we give $\pi^\vee$ the induced filtration $F$, 
    and hence $\gr P_1\to\gr P_0\to\gr_F\pi^\vee\to0$ is exact by \cite[Thm.~I.4.2.4(1)]{LvO96}. 
    So we have $\gr_F\pi^\vee\cong H_0(\gr P_\bullet)\cong N$ and hence $Z(N)=Z(\gr_F\pi^\vee)=Z(\gr_\fm\pi^\vee)$ by Lemma~\ref{cycle}, as desired. 

    Now we will define the filtration $F_i$ on $P_i$. 
    The morphism $\pi^\vee\to\tau^\vee$ naturally extends to a morphism between resolutions
    $$\phi_\bullet:P_\bullet\to L_\bullet.$$
    For $i\in\{0,1,2\}$, the reduction map 
    $$\overline{\phi_i}:\bF\otimes_\Lambda P_i=\Tor_i^\Lambda(\bF,\pi^\vee)\to\Tor_i^\Lambda(\bF,\tau^\vee)=\bF\otimes_{\Lambda} L_i$$
    is injective by Proposition~\ref{tor-inj}. 
    Then, by Lemma~\ref{inj-of-resolution}, the map $\phi_i$ is injective and $P_i$ is identified with a direct summand of $L_i$. 
    We define $F_i$ to be the induced filtration from $L_i$. 

    Now the conditions in Lemma~\ref{lemma-of-filt} are satisfied by Lemma~\ref{min-of-L}, 
    and $\bF\otimes_\Lambda P_i=\Tor_i^\Lambda(\bF,\pi^\vee)\cong \bigoplus_{\chi\in W_D(\bar{\rho})}(\chi^{-1})^{\oplus m_i}\cong \Tor_i^{\gr \Lambda}(\bF,N)\cong\bF\otimes_{\gr \Lambda}\gr L_i'\cong\bF\otimes_{\Lambda} L_i'$.  
    We conclude that $\gr P_i=\gr L_i'=G'_i$, and hence 
    $H_1(\gr P_\bullet)=0$ and $H_0(\gr P_\bullet)=N$. So we get 
    $$\gr_\fm\pi^\vee\cong N,$$
    and we are done. 
\end{proof}

\section{Verifying the assumptions}\label{sec-5}

In this section we show that the assumptions in Theorem~\ref{main-thm}
are satisfied for $\pi$ coming from global geometry. 

We first provide the global setup. See \cite[\S~5]{DL26} for more details. Let $F/\bQ$ be a totally real field in which $p$ is unramified. 
Let $B$ be a totally definite quaternion algebra with center $F$. 
For each place $w$, we write $B_w$ for $B\otimes_F F_w$. 
We assume that there exists $v\mid p$ such that $D:=B_v$ is non-split, 
and $B_w$ is split for $w\mid p$, $w\ne v$. We write $K=F_v$ and $f:=[F_v:\bQ_p]$, so $K$ is unramified over $\bQ_p$ of degree $f$. 

Let $\bar{r}:G_F\to \GL_2(\bF)$ be a continuous Galois representation, and fix a finite-order character 
$\psi:G_F\to \cO_E^\times$ such that $\det\bar{r}=\overline{\psi}\overline{\omega}$. 
We write $\bar{\rho}_w=\bar{r}|_{G_{F_w}}$ and $\psi_w=\psi|_{G_{F_w}}$ for a finite place $w$ of $F$. 
We choose a place $v\mid p$ and write $\bar{\rho}=\bar{\rho}_v$. 
Let $S$ be the finite set of places at which $B$ is ramified or $\bar{r}$ is ramified.  
We make the following assumptions:

(1) $\bar{r}|_{G_{F(\zeta_p)}}$ is absolutely irreducible. 

(2) For each $w\mid p$, $\bar{\rho}_w$ is $1$-generic, and $\bar{\rho}$ is $9$-generic. 
In particular, $\bar{r}$ is ramified at each $w\mid p$. 

(3) For each $w\in S$ not lying over $p$, the universal framed deformation ring $R_{\bar{\rho}_w}^\psi$ of $\bar{\rho}_w$ with $\cO_E$-coefficients and determinant $\psi_w$ 
is formally smooth over $\cO_E$. 

\begin{remark}
    (1) We do not need to assume that $\bar{\rho}$ is $12$-generic as in \cite{BHH$^+$23} and \cite{BHH$^+$25c}.  
     In fact, $9$-genericity is enough, as noted in the introduction to \cite{DL26}. 

    (2) The reason that we assume $B_w$ is split for all $w\mid p$ with $w\ne v$ is to ensure that 
    $R_\infty$ below is regular. See the discussion in \cite[Rk.~5.3.1]{DL26}. 
\end{remark}

We also choose an auxiliary place $w_1$ of $F$ as in \cite[\S~6.2]{EGS15} and an eigenvalue $\beta_{w_1}\in\bF$ of $\bar{\rho}_{w_1}(\Fr_{w_1})$. 
Let $\bT^S=\cO_E[T_w,S_w^{\pm 1}:w\notin S\cup\{w_1\}]$ be the universal Hecke algebra. 
Define 
$$\fm_{\bar{r}}=(p,T_w-\bN(w)\tr(\bar{r}(\Fr_w)), S_w-\bN(w)\det(\bar{r}(\Fr_w)):w\notin S\cup\{w_1\}),$$
where $\bN(w)$ is the cardinality of the residue field at $w$. 
For a compact open subgroup $K^v\subset (B\otimes_F\bA^{\infty,v}_F)^\times$ as in \cite[\S~5.2]{DL26}, 
the space of algebraic modular forms 
$$S(K^v):=\varinjlim_{K_v\subset B_v^\times}H^0(B^\times\backslash (B\otimes_F\bA^\infty_F)^\times/K^vK_v,\cL^v\otimes_{\cO_E}\bF)$$
has natural commuting actions of $B_v^\times=D^\times$ and $\bT^S[T_{w_1}]$; 
here $\cL^v$ is a certain $\cO_E[K^v]$-module and $T_{w_1}$ a certain Hecke operator defined in \cite[\S~5.2]{DL26}. 
Let $\fm_{Q_0}=(\fm_{\bar{r}}, T_{w_1}-\beta_{w_1})$ and 
$\pi=S(K^v)[\fm_{Q_0}]$. This is a $D^\times$-representation. 
We will show that $\pi$ satisfies the assumptions of Theorem~\ref{main-thm} with $r=1$, 
so in particular $\gr_\fm\pi^\vee\cong \bigoplus_{\chi\in W_D(\bar{\rho})}\chi^{-1}\otimes R/\fa(\chi)$. 

In \cite{DL26}, the authors proved that such $\pi$ has Gelfand--Kirillov dimension $f$. 
In fact, they constructed a minimal, self-dual patched module $M_\infty$ over $R_\infty[D^\times]$, satisfying 
$$M_\infty/\fm_{R_\infty}\cong \pi^\vee.$$ 
The authors then proved that $[\pi[\fm^3]:\chi]=1=[\pi[\fm]:\chi]$ for those $\chi$ appearing in $\pi[\fm]$. 
In fact, the set $W_{D}(\bar{\rho})$ is defined as $\{\chi:\Hom_{\cO_D^\times}(\chi, \pi)\ne 0\}$
(see \cite[\S~2.3]{CW23}). Thus, we have $\pi[\fm]=\soc_{\cO_D^\times}\pi=\bigoplus_{\chi\in W_D(\bar{\rho})}\chi$ with the correct central character.  
This verifies conditions~\ref{condition-i} and~\ref{condition-ii}. 

As for condition~\ref{condition-iii}, we will follow the method of \cite[\S~2.6]{BHH$^+$25c}. 
The key ingredient is to show that the patched module $M_\infty$ is flat over $R_\infty$. 

\begin{lemma}
    $M_\infty$ is flat over $R_\infty$. 
\end{lemma}

\begin{proof}
    This is essentially the same as \cite[Thm.~8.4.3]{BHH$^+$23}. 
    We sketch the key points. 
    Write $Z_D^2=Z_D\cap U_D^2$. By the construction of $M_\infty$, $M_\infty$ is a finite free module over $S_\infty[[U_D^2/Z_D^2]]$. 
    Here $S_\infty$ is a certain $\cO_E$-subalgebra of $R_\infty$ as in \cite[\S~5.5]{DL26}, 
    and $\dim R_\infty-\dim S_\infty=2f$. 
    By \cite[Cor.~A.29]{GN22}, 
    $M_\infty$ is a Cohen--Macaulay module over $R_\infty[[U_D^2/Z_D^2]]$. 
    Moreover, we have $j_{\bF[[U_D^2/Z_D^2]]}(M_\infty/\fm_{R_\infty})=3f-\dim\pi=2f$ and $j_{R_\infty[[U_D^2/Z_D^2]]}(M_\infty)=2f$. 
    By \cite[Cor.~A.30]{GN22}, $M_\infty$ is flat over $R_\infty$. 
\end{proof}

\begin{theorem}
    Condition~\ref{condition-iii} is satisfied for $\pi$. 
\end{theorem}

\begin{proof}
    Write $\overline{R_\infty}=R_\infty\otimes_{\cO_E} \bF$. 
    Its maximal ideal is generated by a regular sequence $\underline{y}$. 
    Since $M_\infty$ is flat over $R_\infty$, 
    the argument of \cite[Lemma~2.6.2]{BHH$^+$25c} gives
    $$\Ext_{\cO_D^\times/Z_D^1}^i(\chi,\pi)\cong \Tor^{\overline{R_\infty}}_i(\bF, M_\chi)^\vee.$$ 
    Here $M_\chi=M_\infty\otimes_{\cO_E[[\cO_D^\times/Z_D^1]]}\chi$, 
    whose $R_\infty$-action factors through $R_\infty^{\tau(\chi)}$, the deformation ring with Hodge--Tate weights $(0,1)$ and inertial type $\tau(\chi)$. 
    
    If $\chi\notin W_D(\bar{\rho})$, $R_\infty^{\tau(\chi)}$
    is zero by \cite[Thm.~7.1.1]{EGS15}, so $M_\chi=0$. 
    If $\chi\in W_D(\bar{\rho})$, $M_\chi$ is free of rank one over $\overline{R_\infty^{\tau(\chi)}}$ (see the proof of \cite[Prop.~4.1.1]{DL26}). 
    Let $I_\chi$ be the annihilator of $M_\chi$. 
    Using the same argument as in \cite[Prop.~2.6.3]{BHH$^+$25c}, 
    we deduce that $I_\chi$ is generated by a regular sequence of length $2f$, 
    and that $H_i(K_\bullet(\underline{y},\overline{R_\infty}/I_\chi))\cong \wedge^i(\bF^{\oplus 2f})$. 
    Since $\Tor^{\overline{R_\infty}}_i(\bF, M_\chi)$ is computed by $K_\bullet(\underline{y},M_\chi)\cong K_\bullet(\underline{y},\overline{R_\infty}/I_\chi)$, 
    we obtain the desired result about dimensions.
\end{proof}

\end{document}